\documentclass[final]{siamltex}
\usepackage{graphicx}
\usepackage{url}
\usepackage{mathptmx}     
\usepackage[english]{babel}
\usepackage{amssymb,amsmath}
\usepackage{enumitem}
\usepackage{color}
\usepackage{hyperref}
\usepackage[notref,notcite]{showkeys}
\usepackage{epstopdf}
\usepackage[T1]{fontenc}
\newtheorem{remark}[theorem]{Remark}

\newcommand{\e}{\mathrm{e}}
\newcommand{\R}{\mathbb{R}}
\newcommand{\Z}{\mathbb{Z}}
\newcommand{\E}{\mathbb{E}}
\newcommand{\bigo}[1]{\mathcal{O}(#1)}
\providecommand{\norm}[1]{\ensuremath{\lVert#1\rVert}}
\providecommand{\gnorm}[1]{\ensuremath{\lVert#1\rVert_\sigma}}

\newcommand{\dr}{\mathrm{d}r}

\newcommand{\dd}{\mathrm{d}}
\newcommand{\Tr}{\mathrm{Tr}}

\newcommand{\ii}{\ensuremath{\mathrm{i}}}

\title{Exponential integrators for stochastic Schr\"{o}dinger equations driven by Ito noise}

\author{Rikard Anton\thanks{Department of Mathematics and Mathematical
              Statistics, Ume{\aa} University, SE--901~87~Ume{\aa},
              Sweden ({\tt rikard.anton@umu.se}).}
       \and 
    David Cohen\thanks{Department of Mathematics and Mathematical
              Statistics, Ume{\aa} University, SE--901~87~Ume{\aa}, 
              Sweden ({\tt david.cohen@umu.se}). 
              Department of Mathematics, University of Innsbruck, 
              A--6020~Innsbruck, Austria  
              ({\tt david.cohen@uibk.ac.at})} 
}
\begin{document}
\maketitle
\begin{abstract}
We study an explicit exponential scheme for the time discretisation of 
stochastic Schr\"{o}dinger equations driven by additive or multiplicative Ito noise. 
The numerical scheme is shown to converge with strong order $1$ if the noise is additive 
and with strong order $1/2$ for multiplicative noise. In addition, if the noise is additive, 
we show that the exact solutions of our problems satisfy trace formulas for 
the expected mass, energy, and momentum (i.\,e., linear drifts in these quantities). 
Furthermore, we inspect the behaviour 
of the numerical solutions with respect to these trace formulas. 
Several numerical simulations are presented and confirm our theoretical results. 
\end{abstract}



\pagestyle{myheadings}
\thispagestyle{plain}
\markboth{R. Anton and D. Cohen}{Exponential integrators for stochastic Schr\"{o}dinger equations}


\section{Introduction}\label{sect:intro}

We consider temporal discretisations of nonlinear stochastic Schr\"{o}dinger equations driven by Ito noise
\begin{equation}
\begin{aligned}
&\ii \mathrm{d}u  = \Delta u\, \mathrm{d}t + F(x,u)\,\mathrm{d}t+G(u)\,\mathrm{d}W &&\quad \mathrm{in}\: \ 
\mathbb{R}^d\times(0,\infty), \\
&u(\cdot,0) = u_0 &&\quad \mathrm{in}\: \ \mathbb{R}^d,
\end{aligned}
\label{sse}
\end{equation}
where $u=u(x,t)$, and $\ii=\sqrt{-1}$. The product between $G$ and $\dd W$ is of Ito type, 
and further details on $F$ and $G$ and on the dimension $d$ will be specified later. 
The stochastic process $\{W(t)\}_{t\geq 0}$ is a square integrable complex-valued $Q$-Wiener process with respect 
to a normal filtration $\{\mathcal{F}_t\}_{t\geq 0}$ on a filtered probability space 
$(\Omega,\mathcal{F},\mathbb{P},\{\mathcal{F}_t\}_{t\geq 0})$. The regularity of the covariance operator $Q$ will be specified later in the text.
The initial value $u_0$ is an $\mathcal{F}_0$-measurable complex-valued function, 
which will be further specified below.

The Schr\"{o}dinger equation is widely used within physics and takes several different forms depending on the situation. 
It is used in hydrodynamics, nonlinear optics and plasma physics to only mention a few areas. 
In certain physical situations it may be appropriate to incorporate some kind of randomness into the equation. One possibility is to add a driving random force to then obtain an equation of the form \eqref{sse}. 
See for example \cite{AStochasticNonlinearSchrodingerEquation} and references therein for further details.

Stochastic Schr\"odinger equations have received much attention from a more theoretical point 
of view during the last decades. Connected to the present article and without being exhaustive, 
we mention the works \cite{effectonnoise,SNLSEinH1,Theoreticalnumericalaspectsstochasticnonlinear} on Ito problems and 
\cite{AStochasticNonlinearSchrodingerEquation,SNLSEinH1,dbd05,mr07,Theoreticalnumericalaspectsstochasticnonlinear} for the Stratonovich setting. 

It is seldom possible to solve stochastic partial differential equations exactly, and efficient numerical schemes are 
therefore needed. For the time integration of the above stochastic Schr\"odinger equations, we will consider stochastic 
exponential integrators. These numerical methods are explicit and easy to implement, furthermore 
they offer good geometric properties. Exponential integrators are widely used and studied nowadays 
as witnessed by the recent review \cite{ho10} for the time integration of deterministic problems. 
Applications of such schemes to the deterministic (nonlinear) Schr\"odinger equation can 
be found in, for example, \cite{MR1735097,bos06,MR2346679,Celledoni2008,d09,cg,cgp15} and references therein. 
These numerical methods were recently investigated for stochastic parabolic partial differential equations  
in, for example, \cite{Lord2004,MR2578878,MR3047942} and for the stochastic 
wave equations in \cite{cls12s,Wang2014,cqs14s,aclw15s}.

We now review previous works on temporal discretisations of stochastic Schr\"{o}dinger 
equations. In \cite{Weakandstrongorderofconv} a Crank-Nicolson scheme 
is studied for the equation with nonlinearity $F(u)$. 
First order of convergence is obtained in the case of additive noise, and with multiplicative Ito noise the convergence rate is one half. 
Observe that this numerical scheme is implicit. A stochastic Schr\"odinger equation with Stratonovich noise is considered 
in \cite{Asemi-discretescheme}, where, again, a Crank-Nicolson scheme is studied for the equation with nonlinearity 
$F(x,u)=\lambda|u|^{2\sigma}u$, with $\lambda=\pm 1$ and $\sigma>0$. The authors prove convergence to the exact solution and mass preservation of the scheme. Further, in \cite{Amass-preservingsplittingscheme} a mass-preserving splitting scheme for equation \eqref{sse} with $F(x,u)=V(x)u$ and $G(u)=u$ is considered. 
The noise is of Stratonovich type and first order convergence is obtained. 
In \cite{liu13b}, $V(x)$ is replaced by $|u|^2$ and first order convergence is again obtained. 
Still in the Stratonovich setting, \cite{jwh13} derives multi-symplectic schemes for stochastic Schr\"odinger equations. 
We finally mention \cite{Numericalsimulationfocusingstochastic, Barton-Smith_et_al-2005-Numerical_Methods_for_Partial_Differential_Equations}, 
in which thorough numerical simulations are presented for both additive noise and multiplicative Stratonovich noise.

In the present work we show that
\begin{itemize}
\item the exponential integrator applied to the linear stochastic Schr\"odinger equation 
with additive noise converges strongly with order $1$ and satisfies exact trace formulas for the mass, the energy, and for the momentum; 
\item the exponential integrator applied to the stochastic Schr\"odinger equation with a multiplicative potential 
and additive noise converges with strong order $1$, but has a small error in the trace formulas for the mass and energy; 
\item the exponential integrator applied to stochastic Schr\"odinger equations driven by multiplicative 
Ito noise strongly converges with order $1/2$.
\end{itemize}

We begin the exposition by introducing some notations and useful results that we will use in our proofs. 
After that we will follow a similar approach as in \cite{Weakandstrongorderofconv}. 
That is, we will begin by analysing the numerical method applied to the linear 
Schr\"odinger equation with additive noise in Section~\ref{sec-linadd}. Then we study stochastic Schr\"odinger 
equations with a multiplicative potential in Section~\ref{sec-potadd} and finally we consider 
the stochastic Schr\"odinger equation with a multiplicative potential 
and multiplicative noise in Section~\ref{sec-mult}. 
For each of the above problems, we analyse the speed of convergence of 
the exponential methods (in the strong sense) and for additive problems 
we show some trace formulas (such results could be interpreted as weak error estimates). 
Various numerical experiments accompany the presentation and illustrate 
the main properties of these exponential methods 
when applied to stochastic Schr\"odinger equations driven by Ito noise.


\section{Notations and some useful results}\label{sec-not}
Given two separable Hilbert spaces $H_1$ and $H_2$ with norms $\norm{\cdot}_{H_1}$ and $\norm{\cdot}_{H_2}$ respectively, 
we denote the space of bounded linear operators from $H_1$ to $H_2$ by $\mathcal{L}(H_1,H_2)$. 
We denote by $\mathcal{L}_2(H_1,H_2)$ the set of Hilbert-Schmidt operators from $H_1$ to $H_2$ with norm
$$
\norm{\Phi}_{\mathcal{L}_2(H_1,H_2)}:=\left(\sum_{k=1}^\infty\norm{\Phi e_k}_{H_2}^2\right)^{1/2},
$$
where $\{e_k\}_{k=1}^{\infty}$ is any orthonormal basis of $H_1$.
Furthermore, for Hilbert spaces $H_1$, $H_2$, and $H_3$ 
we have (see the proof of Lemma~$2.1$ in \cite{AStochasticNonlinearSchrodingerEquation}) 
that if $S\in\mathcal{L}_2(H_1,H_2)$ and $T\in\mathcal{L}(H_2,H_3)$, then $TS\in\mathcal{L}_2(H_1,H_3)$ and
\begin{align}\label{ineq1}
\norm{TS}_{\mathcal{L}_2(H_1,H_3)}\leq \norm{T}_{\mathcal{L}(H_2,H_3)}\norm{S}_{\mathcal{L}_2(H_1,H_2)}.
\end{align}
Also, we consider $L^2(\mathbb{R}^d)$ the space of square integrable functions on $\mathbb{R}^d$ with inner product
\begin{align}\label{innerp}
(u,v)=\mathrm{Re}\int_{\mathbb{R}^d}u\bar{v}\,\dd x.
\end{align}
For $\sigma\in\mathbb{R}$, we further denote the fractional Sobolev space of order $\sigma$ by $H^\sigma=H^\sigma(\R^d)$ 
with norm $\gnorm{\cdot}$. To make the notations cleaner and 
more readable we will use the following shorter notations
$$
L^2=L^2(\mathbb{R}^d),\quad \mathcal{L}(H_1)=\mathcal{L}(H_1,H_1),\quad \mathcal{L}_2^\sigma=\mathcal{L}_2(L^2,H^\sigma).
$$
We note that the operator $-\ii\Delta$, appearing in the stochastic Schr\"odinger equation, 
is the generator of a semigroup of isometries of bounded linear operators 
$S(t)=\e^{-\ii t\Delta}$. We will make use of the following result
\begin{lemma}\label{lemma1}
(See e.g.~\cite[Lemma $3.2$]{Amass-preservingsplittingscheme})
Consider any $\sigma\geq 0$ and $S(t)=\e^{-\ii t\Delta}$, $t\geq 0$, then, for $w\in H^\sigma$, one has 
$$
\gnorm{S(t)w}=\gnorm{w},
$$
and, for $\Delta w\in H^\sigma$ and any $t\geq 0$,
$$
\gnorm{\Delta(S(t)w)}=\gnorm{\Delta w}, \quad\norm{(S(t)-I)w}_\sigma\leq t\norm{\Delta w}_{\sigma}.
$$ 
\end{lemma}
As a consequence of the above lemma and \eqref{ineq1}, we have, for $\Phi\in\mathcal{L}_2^{\sigma+2}$ and any $t\geq 0$,
$$
\norm{(S(t)-I)\Phi}_{\mathcal{L}_2^{\sigma}}\leq t\norm{\Phi}_{\mathcal{L}_2^{\sigma+2}}.
$$
We will also use the following result.
\begin{lemma}\label{lemma2}
For any $\sigma\geq 0$ and $t\geq 0$, we have
$$\norm{S(t)-I}_{\mathcal{L}(H^{\sigma+1},H^\sigma)}\leq Ct^{1/2},$$
for a constant $C$.
\end{lemma}

The proof of this result is very similar to the proof of Lemma~\ref{lemma1}.

Finally, throughout the paper, $C$ (and $C_1$, $C_2$ etc.) will denote a generic constant which may change from line to line. 


\section{The linear Schr\"odinger equation with additive noise}\label{sec-linadd}
In this section, we study time discretisations of the linear stochastic Schr\"odinger equation
\begin{equation}
\begin{aligned}
&\ii\mathrm{d}u-\Delta u\,\mathrm{d}t=\mathrm{d}W &&\quad \mathrm{in}\: \ 
\mathbb{R}^d\times(0,\infty), \\
&u(0)=u_0 &&\quad \mathrm{in}\: \ \mathbb{R}^d,
\end{aligned}
\label{linSch}
\end{equation}
where $u_0$ is an $\mathcal{F}_0$-measurable random variable and the noise $\{W(t)\}_{t\geq0}$ 
is a square integrable complex-valued $Q$-Wiener process with respect to the filtration. 
In this section, we have no restrictions on the dimension $d$.

Global existence and uniqueness in $H^\sigma$, $\sigma\geq 0$, of the solution 
to equation \eqref{linSch} is guaranteed if $u_0\in H^\sigma$ a.s. and $Q^{1/2}\in\mathcal{L}_2^\sigma$. The proof follows the same line 
as the proof of Theorem $7.4$ in \cite{DaPrato}.
The mild solution of \eqref{linSch} reads 
$$
u(t)=S(t)u_0-\ii\int_0^t S(t-r)\,\dd W(r),
$$
where we recall the notation $S(t)=\e^{-\ii t\Delta}$. 

We now consider the time integration of the above problem. This is done as follows. 
Let $T>0$ be a fixed time horizon and $N>0$ be an integer. We first divide the interval $[0,T]$ 
into subintervals $0=t_0<t_1<\ldots<t_{N-1}<t_N=T$ of equal length $k$ so that $t_n=nk$. 
An exponential integrator with step size $k$ is now derived by approximating the above stochastic integral, 
in the mild solution, at the left end point. We thus obtain a numerical approximation $u^n$ of the exact solution 
$u(t_n)$ of \eqref{linSch}:
\begin{align}\label{LinearScheme}
u^n &= S(k)u^{n-1}-\ii S(k)\Delta W^{n-1}=S(t_n)u_0-\ii\sum_{j=0}^{n-1}\int_{t_j}^{t_{j+1}}S(t_n-t_j)\,\mathrm{d}W(r),
\end{align}
where $\Delta W^{n-1}=W^n-W^{n-1}=W(t_n)-W(t_{n-1})$ denotes Wiener increments. 
We call this explicit numerical method an exponential integrator.

\subsection{Error estimates}
This subsection presents a result on the error of the exponential integrator \eqref{LinearScheme} 
when applied to the linear stochastic Schr\"odinger equation \eqref{linSch}. 
These error estimates are given in the next theorem.
\begin{theorem}\label{maxlinear}
Let $\sigma\geq 0$, $p\in\mathbb{N}$. Recall that $u^n$ is the numerical approximation given 
by the exponential integrator \eqref{LinearScheme} of the exact solution $u(t_n)$ 
to the linear stochastic Schr\"odinger equation driven by a $Q$-Wiener process \eqref{linSch}.  
Assume that $u_0\in H^\sigma$ a.s., and $Q^{1/2}\in\mathcal{L}_2^{\sigma+2}$. 
Then there exists a constant $C$ such that
$$
\E\Big[\max_{n=1,\ldots,N}\norm{u^n-u(t_n)}^{2p}_\sigma\Big]\leq Ck^{2p}\norm{Q^{1/2}}_{\mathcal{L}_{2}^{\sigma+2}}^{2p}.
$$
\end{theorem}

\begin{proof}
We have
\begin{align*}
u^n-u(t_n)&=\ii\sum_{j=0}^{n-1}\int_{t_j}^{t_{j+1}}(S(t_n-r)-S(t_n-t_j))\,\dd W(r)\\
&=\ii\int_0^{t_n}\left(S(t_n-r)-S\left(t_n-[r/k]k\right)\right)\,\dd W(r),
\end{align*}
where $[r/k]$ denotes the integer part of $r/k$. The last equality is indeed correct since if $t_j\leq r< t_{j+1}$, then $j\leq r/k< j+1$ so that $[r/k]=j$ and $[r/k]k=jk=t_j$.
Using Burkholder's inequality \cite[Lemma~7.2]{DaPrato},  
and Lemma~\ref{lemma1}, we have
\begin{align*}
\E\Bigl[
\max_{n=1,\ldots,N}\norm{u^n-u(t_n)}^{2p}_\sigma\Bigr]&\leq
\E\Bigl[\sup_{t\in[0,T]}\norm{\ii\int_{0}^{t}\left(S(t-r)-S\left(t-[r/k]k\right)\right)\,\dd W(r)}^{2p}_\sigma\Bigr]\\
&\leq C\E\left[\left(\int_{0}^{T}\norm{\left(S([r/k]k-r)-I\right)Q^{1/2}}_{\mathcal{L}_2^\sigma}^2\,\dd r\right)^p\right]\\
&= C\E\left[\left(\sum_{j=0}^{N-1}\int_{t_j}^{t_{j+1}}\norm{\left(S(t_j-r)-I\right)Q^{1/2}}_{\mathcal{L}_2^\sigma}^2\,\dd r\right)^p\right]\\
&\leq C\E\left[\left(\sum_{j=0}^{N-1}\int_{t_j}^{t_{j+1}}|r-t_j|^2\norm{Q^{1/2}}_{\mathcal{L}_2^{\sigma+2}}^2\,\dd r\right)^p\right]\\
&\leq Ck^{2p}\norm{Q^{1/2}}_{\mathcal{L}_{2}^{\sigma+2}}^{2p}.
\end{align*}
\end{proof}

\subsection{A trace formula for the mass}\label{sect:massL}
Under appropriate boundary conditions, 
for example periodic boundary conditions or homogeneous Dirichlet boundary conditions, 
the mass (also called $L^2$-norm or density) 
$$
M(u):=\int |u|^2\,\dd x
$$
of the deterministic linear Schr\"odinger equation $\ii \frac{\partial u}{\partial t}-\Delta u=0$ is a conserved quantity. 
In the stochastic case, one immediately gets a trace formula for the mass of the exact solution as stated below.  

\begin{proposition}\label{prop:MassL}
Assume that the initial data is such that $\E[M(u_0)]$ is finite and that the covariance operator $Q$ is trace-class. 
Then the exact solution of the linear Schr\"{o}dinger equation with additive noise \eqref{linSch} 
satisfies the trace formula for the expected mass
$$
\E[M(u(t))]=\E[\norm{u(t)}_{L^2}^2]=\E[M(u_0)]+t\Tr(Q)\quad\text{for all time }  t.
$$
\end{proposition}

\begin{proof}
Using Lemma~\ref{lemma1} and the fact that the stochastic integrals are normally distributed with mean $0$, we get
\begin{align*}
\E[\norm{u(t)}_{L^2}^2]&=\E\left[\norm{S(t)u_0}_{L^2}^2+\left(S(t)u_0,-\ii\int_0^t S(t-r)\,\dd W(r)\right)\right.\\
&\quad \left. +\left(-\ii\int_0^t S(t-r)\,\dd W(r),S(t)u_0\right)+\norm{\int_0^t S(t-r)\,\dd W(r)}_{L^2}^2\right]\\
&=\E[\norm{S(t)u_0}_{L^2}^2]+\int_0^t\E[\norm{S(t-r)Q^{1/2}}_{\mathcal{L}_2^0}^2]\,\dr\\
&=\E[\norm{u_0}_{L^2}^2]+t\Tr(Q).
\end{align*}
\end{proof}
\begin{remark}
We would like to point out that, as in the case of the linear stochastic wave equation treated in \cite{cls12s}, 
an alternative proof of the above result can be obtained using Ito's formula 
(see, for example \cite[Theorem 4.17]{DaPrato}). This remark is also valid for the other trace formulas given below.
\end{remark}

Our exponential integrator does indeed satisfy this trace formula for the mass as well, as seen in the next result. 
\begin{proposition}\label{prop:MassLNum}
With the same assumptions as in Proposition \ref{prop:MassL}, the stochastic exponential integrator 
\eqref{LinearScheme} satisfies the following trace formula for the mass
\begin{align*}
\E[M(u^n)]&=\E[M(u^{n-1})]+k\Tr(Q)\\
&=\E[M(u_0)]+t_n\Tr(Q)\quad\text{for all}\quad t_n=nk.
\end{align*}
\end{proposition}

\begin{proof}
Similarly to the proof of Proposition \ref{prop:MassL}, we have
\begin{align*}
\E[M(u^n)]&=\E[M(u^{n-1})]+\E[\norm{\int_{t_{n-1}}^{t_n}S(k)\,\dd W(r)}_{L^2}^2]\\
&=\E[M(u^{n-1})]+k\Tr(Q).
\end{align*}
A recursion concludes the proof.
\end{proof}

We would like to examine the behaviour of the Euler-Maruyama scheme, 
the backward Euler-Maruyama scheme, and 
the midpoint rule with respect to the trace formula for the mass. Since we will consider periodic domains and 
pseudospectral discretisations in the numerical experiments presented below, we will examine this case first. 

A pseudospectral spatial discretisation of \eqref{linSch} with a noise given by the representation 
$$
W(x,t)=\sum_{n\in\Z}\lambda_n^{1/2}\beta_n(t)e_n(x),
$$ 
where $\{\beta_n(t)\}_{n\in\Z}$ are i.i.d Brownian motions, $\lambda_n$ are eigenvalues of $Q$, 
and $\{e_n(x)\}_{n\in\Z}=\{\frac{1}{\sqrt{2\pi}}\e^{\ii nx}\}_{n\in\Z}$ is an orthonormal basis 
of $L^2(0,2\pi)$, will lead to the following system of decoupled stochastic differential equations for 
the Fourier coefficients $y_k$ of the exact solution
$$
\ii\text d y_k=-k^2y_k\,\text dt+\lambda_k^{1/2}\,\text d\beta_k.
$$
This motivates us to consider the scalar test problem 
(with a standard Brownian motion $\beta$ and real numbers $a,b$)   
\begin{align}\label{testSDE}
\ii\text d y=ay\,\text dt+b\,\text d\beta
\end{align}
and the quantity equivalent to the mass is thus the second moment $\E[|y|^2]$. 
We thus have the trace formula for the exact solution
$$
\E[|y(t)|^2]=\E[|y(0)|^2]+b^2t\quad\text{for all times}\quad t.
$$
The following result states that the above classical numerical methods do not preserve the trace formula for 
this simple test problem, and in particular, are not suited when applied to 
pseudospectral discretisations of linear stochastic Schr\"odinger equations.
\begin{proposition}\label{prop:MassNum}
Consider a pseudospectral discretisation of the stochastic Schr\"odinger equation \eqref{linSch} 
with a trace-class noise and periodic boundary conditions yielding to equations of  
the form \eqref{testSDE}. We have the following results:
\begin{enumerate}
\item The Euler-Maruyama scheme 
$$
y^{n+1}=y^n-\ii ka y^n-\ii b\Delta W^n
$$
produces a second moment that grows exponentially with time 
$$
\E[|y^n|^2]\geq\e^{(\frac12ka^2)t_n}\E[|y^0|^2]\quad\text{for}\quad t_n=nk.
$$
\item The backward Euler-Maruyama scheme 
$$
y^{n+1}=y^n-\ii ka y^{n+1}-\ii b\Delta W^n
$$
produces a second moment that grows at a slower rate than the exact solution
$$
\E[|y^n|^2]\leq\E[|y^0|^2]+\frac{b^2}{a^2k}\quad\text{for all}\quad n\geq0,
$$
hence $\displaystyle\lim_{t_n\to\infty}\bigl(\E[|y^n|^2]/t_n\bigr)=0$.
\item The midpoint rule 
$$
\ii \frac{y^{n+1}-y^{n}}{k}-a\frac{y^{n+1}+y^n}2=b \Delta W^n
$$
produces a second moment that underestimate the linear drift of the exact solution
$$
\E[|y^n|^2]=\E[|y^0|^2]+\frac{b^2t_n}{1+\frac{a^2k^2}{2}}\quad\text{for}\quad t_n=nk.
$$
\end{enumerate}
\end{proposition}
\begin{proof}
The proof of this proposition is an easy adaptation of the results presented in \cite{smh04}. 
We start by looking at the behaviour of the Euler-Maruyama scheme and compute, 
using properties of the Wiener increments, 
\begin{align*}
\E[|y^{n+1}|^2]&=(1+a^2k^2)\E[|y^n|^2]+b^2k\geq (1+a^2k^2)\E[|y^n|^2]
\geq (1+a^2k^2)^{n+1}\E[|y^0|^2]\\
&\geq \e^{(\frac12ka^2)t_{n+1}}\E[|y^0|^2].
\end{align*}
For the backward Euler-Maruyama scheme we obtain, using a geometric series, 
\begin{align*}
\E[|y^{n+1}|^2]&=\frac{1}{1+a^2k^2}\Bigl(\E[|y^n|^2]+b^2k\Bigr)\leq\E[|y^n|^2]+\frac{b^2k}{1+a^2k^2}
\leq \E[|y^0|^2]+\frac{b^2}{a^2k}.
\end{align*}
Finally, for the midpoint rule, we have
\begin{align*}
\E[|y^{n+1}|^2]=\E[|y^{n}|^2]+\frac{b^2k}{1+\frac{a^2k^2}{2}}=\E[|y^{0}|^2]+\frac{b^2t_{n+1}}{1+\frac{a^2k^2}{2}}.
\end{align*}
\end{proof}

For homogeneous Dirichlet boundary conditions a similar result holds.
\begin{proposition}\label{prop:MassNum2}
Consider the stochastic Schr\"odinger equation \eqref{linSch} with a trace-class noise and 
homogeneous Dirichlet boundary conditions. We have the following results:
\begin{enumerate}
\item The Euler-Maruyama scheme 
$$
u^{n+1}=u^n-\ii k\Delta u^n-\ii\Delta W^n
$$
produces a numerical trace formula for the mass that grows exponentially with time 
$$
\E[M(u^n)]\geq\e^{(\frac12k\lambda_1)t_n}\E[M(u^{0})]\quad\text{for}\quad t_n=nk,
$$
where $\lambda_1$ denotes the smallest eigenvalue of the Laplacian.
\item The backward Euler-Maruyama scheme 
$$
u^{n+1}=u^n-\ii k\Delta u^{n+1}-\ii\Delta W^n
$$
produces a numerical trace formula for the mass that grows at a slower rate than the exact solution
$$
\E[M(u^n)]\leq \E[M(u^0)]+\frac{\Tr(Q)}{\lambda_1^2 k},\quad\text{for all}\quad n\geq 0,
$$
where $\lambda_1$ is the smallest eigenvalue of the Laplacian. Thus $\displaystyle\lim_{t_n\to\infty}\bigl(\E[M(u^n)]/t_n\bigr)=0.$
\item The midpoint rule \cite{Weakandstrongorderofconv} 
$$
\ii \frac{u^{n+1}-u^{n}}{k}-\Delta \frac{u^{n+1}+u^n}2=\Delta W^n
$$
underestimate the expected mass:
$$\E[M(u^n)]\leq \E[M(u^0)]+\frac{t_n}{1+\frac{k^2\lambda_1}{4}}\Tr(Q),\quad\text{for all}\quad n\geq 0,$$
where $\lambda_1$ is the smallest eigenvalue of the Laplacian.
\end{enumerate}
\end{proposition}
\begin{proof}
\begin{enumerate}
\item We start by looking at the behaviour of the Euler-Maruyama scheme and compute 
\begin{align*}
\E[M(u^n)]=\E[M(u^{n-1})]+k^2\E[\norm{\Delta u^{n-1}}^2]+k\Tr(Q)\geq \E[M(u^{n-1})]+
k^2\lambda_1^2\E[\norm{u^{n-1}}^2]+k\Tr(Q),
\end{align*}
where $\lambda_1$ denotes the smallest eigenvalue of the Laplacian. Hence, one obtains 
\begin{align*}
\E[M(u^n)]\geq(1+k^2\lambda_1^2)\E[M(u^{n-1})]\geq(1+k^2\lambda_1^2)^n\E[M(u^{0})]\geq\e^{(\frac12k\lambda_1)t_n}\E[M(u^{0})].
\end{align*}
\item For the backward Euler-Maruyama scheme we obtain
$$u^{n+1}+\ii k\Delta u^{n+1}=u^n-\ii\Delta W^n.$$
Taking the norm and expectation we have
$$\E[\norm{u^{n+1}}^2]+k^2\E[\norm{\Delta u^{n+1}}^2]=\E[\norm{u^n}^2]+\E[\norm{\Delta W^n}^2].$$
Using that the eigenvalues are positive and increasing, and using a similar argument as in the proof of Proposition \ref{prop:MassNum}, we get
\begin{align*}
\E[\norm{u^{n+1}}^2]&\leq \frac{1}{1+k^2\lambda_1^2}\left(\E[\norm{u^n}^2]+k\Tr(Q)\right)\\
&\leq \E[\norm{u^0}^2]+\frac{\Tr(Q)}{k\lambda_1^2},
\end{align*}
where $\lambda_1$ is the smallest eigenvalue of the Laplacian.
\item Let $(\lambda_j,e_j)$ be the eigenpairs of the Laplace operator and let $(\alpha_l,f_l)$ be the eigenpairs of the covariance operator $Q$. Writing $u(t)=\sum_{j=1}^\infty c_j(t)e_j$ we have that the expected mass is given by
$$\E[M(u(t))]=\sum_{j=1}^\infty \E[|c_j(t)|^2].$$
The midpoint rule thus becomes
$$\sum_{j=1}^\infty \left(1+\frac{\ii k\lambda_j}{2}\right)c_j^{n+1}e_j = \sum_{j=1}^\infty \left(1-\frac{\ii k\lambda_j}{2}\right)c_j^{n}e_j-\ii\sum_{j=1}^\infty\sum_{l=1}^\infty\alpha_l\Delta\beta_l^n(f_l,e_j)e_j.$$
Thus, for every $j\geq 1$ we have
$$\left(1+\frac{\ii k\lambda_j}{2}\right)c_j^{n+1}=\left(1-\frac{\ii k\lambda_j}{2}\right)c_j^{n}-\ii\sum_{l=1}^\infty\alpha_l\Delta\beta_l^n(f_l,e_j).$$
Multiplying both sides with the conjugate and taking the expectation, we have
\begin{align*}
\left(1+\frac{k^2\lambda_j^2}{4}\right)\E[|c_j^{n+1}|^2]&=\left(1+\frac{k^2\lambda_j^2}{4}\right)\E[|c_j^{n}|^2]+\sum_{l=1}^\infty|\alpha_l|^2\E[|\Delta\beta_l^n|^2]|(f_l,e_j)|^2\\
&=\left(1+\frac{k^2\lambda_j^2}{4}\right)\E[|c_j^{n}|^2]+k\sum_{l=1}^\infty|\alpha_l|^2|(f_l,e_j)|^2.
\end{align*}
Using that $\{\lambda_j\}_{j=1}^\infty$ is a positive and increasing sequence, we get
\begin{align*}
\E[|c_j^{n+1}|^2]&=\E[|c_j^n|^2]+\frac{k}{1+\frac{k^2\lambda_j}{4}}\sum_{l=1}^\infty|\alpha_l|^2|(f_l,e_j)|^2\\
&=\E[|c_j^0|^2]+\frac{t_{n+1}}{1+\frac{k^2\lambda_j}{4}}\sum_{l=1}^\infty|\alpha_l|^2|(f_l,e_j)|^2\\
&\leq \E[|c_j^0|^2]+\frac{t_{n+1}}{1+\frac{k^2\lambda_1}{4}}\sum_{l=1}^\infty|\alpha_l|^2|(f_l,e_j)|^2.
\end{align*}
Now using Parseval's identity, we finally obtain 
\begin{align*}
\E[M(u^{n+1})]=\sum_{j=0}^\infty \E[|c_j^{n+1}|^2]&\leq \sum_{j=0}^\infty \E[|c_j^0|^2]+\frac{t_{n+1}}{1+\frac{k^2\lambda_1}{4}}\sum_{l=1}^\infty|\alpha_l|^2\sum_{j=0}^\infty |(f_l,e_j)|^2\\
&=\sum_{j=0}^\infty \E[|c_j^0|^2]+\frac{t_{n+1}}{1+\frac{k^2\lambda_1}{4}}\sum_{l=1}^\infty|\alpha_l|^2\norm{f_i}^2\\
&=\sum_{j=0}^\infty \E[|c_j^0|^2]+\frac{t_{n+1}}{1+\frac{k^2\lambda_1}{4}}\sum_{l=1}^\infty|\alpha_l|^2\\
&=\E[M(u^{0})]+\frac{t_{n+1}}{1+\frac{k^2\lambda_1}{4}}\Tr(Q).
\end{align*}
\end{enumerate}
\end{proof}

\subsection{A trace formula for the energy}\label{sect:tr}
Again, under appropriate boundary conditions, for example periodic boundary conditions 
or homogeneous Dirichlet boundary conditions, 
it is well known that the energy 
$$
H(u(t)):=\frac{1}{2}\int|\nabla u|^2\,\dd x
$$
of the deterministic linear Schr\"{o}dinger equation $\ii \frac{\partial u}{\partial t}-\Delta u=0$ 
remains constant along the exact solution, see for example~\cite{cazenaveAnIntro}. 
When additive noise is introduced into the problem, we get a linear drift in the expected value of the total energy 
as stated in the following result. 
\begin{proposition}\label{prop:traceL}
Assume that the initial data is such that $\E[H(u_0)]$ is finite, and that $\norm{\nabla Q^{1/2}}_{\mathcal{L}_2^0}$ is bounded. 
Then the exact solution of the linear Schr\"{o}dinger equation with additive noise \eqref{linSch} 
satisfies the trace formula for the expected energy
\begin{align*}
\E\bigl[H(u(t))\bigr]=\E\bigl[H(u_0)\bigr]+\frac{t}{2}\Tr(\nabla Q\nabla)\quad\text{for all time}\quad t. 
\end{align*}
\end{proposition}

\begin{proof}
Using Lemma~\ref{lemma1} and the fact that the Ito integral is normally distributed with mean $0$, we have
\begin{align*}
\E[\norm{\nabla u(t)}_{L^2}^2]&= \E[\norm{\nabla(S(t)u_0)-\ii\int_0^t \nabla S(t-r)\,\mathrm{d}W(r)}_{L^2}^2]\\
&=\E\left[\norm{\nabla(S(t)u_0)}_{L^2}^2+\norm{\int_0^t \nabla S(t-r)\,\mathrm{d}W(r)}_{L^2}^2\right.\\
&\quad-\left.\left(\nabla(S(t)u_0),\ii\int_0^t\nabla S(t-r)\mathrm{d}W(r)\right)-
\left(\ii\int_0^t\nabla S(t-r)\mathrm{d}W(r),\nabla(S(t)u_0)\right)\right]\\
&=\E[\norm{\nabla u_0}_{L^2}^2]+\int_0^t\E[\norm{\nabla(S(t-r)Q^{1/2})}^2_{\mathcal{L}_2^0}]\,\dr\\
&=\E[\norm{\nabla u_0}_{L^2}^2]+t\Tr(\nabla Q\nabla).
\end{align*}
\end{proof}

We now show that our exponential integrator satisfies the very same energy trace formula as 
the exact solution to the linear stochastic Schr\"odinger equation \eqref{linSch}.
\begin{proposition}\label{traceLEXP}
With the same assumptions as in Proposition \ref{prop:traceL}, the exponential integrator 
\eqref{LinearScheme} satisfies the following trace formula
\begin{align*}
\E\bigl[H(u^n)\bigr]&=\E\bigl[H(u^{n-1})\bigr]+\frac{k}{2}\Tr(\nabla Q\nabla)\\
&=\E\bigl[H(u_0)\bigr]+\frac{t_n}{2}\Tr(\nabla Q\nabla)\quad\text{for all}\quad t_n=nk. 
\end{align*}
\end{proposition}

\begin{proof}
Similarly to the proof of the previous proposition, we have
\begin{align*}
\E\bigl[H(u^n)\bigr]&=\frac{1}{2}\E[\norm{\nabla u^n}_{L^2}^2]\\
&=\frac{1}{2}\E[\norm{\nabla(S(k)u^{n-1})-\ii\int_{t_{n-1}}^{t_n}\nabla S(k)\,\dd W(r)}_{L^2}^2]\\
&=\frac{1}{2}\E[\norm{\nabla u^{n-1}}_{L^2}^2]+\frac{1}{2}\int_{t_{n-1}}^{t_n} \E[\norm{\nabla(S(k)Q^{1/2})}_{\mathcal{L}_2^0}^2]\,\dd r\\
&=\E\bigl[H(u^{n-1})\bigr]+\frac{k}{2}\Tr(\nabla Q\nabla).
\end{align*}
\end{proof}

\subsection{A formula for the momentum}\label{sect:momentum}
In the deterministic case, one has an additional conserved quantity, the momentum
$$
p(u):=\ii\int (u\nabla\bar u-\bar u\nabla u)\,\dd x.
$$
The next result investigate the behaviour of this quantity in the stochastic case. 
\begin{proposition}\label{prop:momentum}
Assume that $\E[p(u_0)]<\infty$ and $Q^{1/2}\in\mathcal{L}_2^1$. 
Then the exact solution of the linear Schr\"{o}dinger equation with additive noise \eqref{linSch} 
exhibits the following formula for the expected momentum
$$
\E[p(u(t))]=\E[p(u_0)]-2t\mathrm{Im}\left< Q^{1/2},\nabla Q^{1/2}\right>_{\mathcal{L}_2^0}  \quad\text{for all time}\quad t.
$$
Here, we denote by $\left<\cdot,\cdot\right>$ the usual $L^2$ inner product
\begin{equation*}
\left< u,v\right>:=\int_{\mathbb{R}^d}u\bar{v}\,\dd x,
\end{equation*}
for $u,v\in L^2$. And further denote the Hilbert-Schmidt inner product from $L^2$ to $L^2$ 
with the above inner product by $\left< \cdot,\cdot\right>_{\mathcal{L}_2^0}$. 
\end{proposition}
\begin{proof}
The momentum can be written as
$$p(u(t))=\ii\left< u(t),\nabla u(t)\right>-\ii\left< \nabla u(t), u(t)\right>.$$
Using the mild solution, we have
\begin{align*}
\left< u(t),\nabla u(t)\right>&=\left< S(t)u_0-\ii\int_0^t S(t-r)\,\dd W(r),\nabla S(t)u_0-\ii\int_0^t\nabla S(t-r)\,\dd W(r)\right>\\
&=\left< S(t)u_0,\nabla S(t)u_0\right>+\left< S(t)u_0,-\ii\int_0^t S(t-r)\nabla\dd W(r)\right>\\
&\quad+\left< -\ii\int_0^t S(t-r)\dd W(r),S(t)\nabla u_0\right>\\
&\quad+\left< -\ii\int_0^t S(t-r)\dd W(r),-\ii\int_0^t S(t-r)\nabla\dd W(r)\right>.
\end{align*}
Taking the expectation, the terms containing one integral is zero. Using also the Ito isometry we obtain
\begin{align*}
\E[\left< u(t),\nabla u(t)\right>]&=\E[\left< S(t)u_0,\nabla S(t)u_0\right>]+\E\left[\left< \int_0^t S(t-r)\dd W(r),\int_0^t S(t-r)\nabla\dd W(r)\right>\right]\\
&=\E[\left< u_0,\nabla u_0\right>]+\E\left[\int_0^t\left< S(t-r)Q^{1/2},S(t-r)\nabla Q^{1/2}\right>_{\mathcal{L}_2^0}\,\dd r\right]\\
&=\E[\left< u_0,\nabla u_0\right>]+\E\left[\int_0^t\left< Q^{1/2},\nabla Q^{1/2}\right>_{\mathcal{L}_2^0}\,\dd r\right]\\
&=\E[\left< u_0,\nabla u_0\right>]+t\left< Q^{1/2},\nabla Q^{1/2}\right>_{\mathcal{L}_2^0}.
\end{align*}
Similarly, we have 
\begin{align*}
\E[\left< \nabla u(t), u(t)\right>]&=\E[\left< \nabla u_0, u_0\right>]+t\left< \nabla Q^{1/2},Q^{1/2}\right>_{\mathcal{L}_2^0}.
\end{align*}
For the expected momentum we finally obtain
\begin{align*}
\E[p(u(t))]&=\E[\ii\left< u(t),\nabla u(t)\right>-\ii\left< \nabla u(t), u(t)\right>]\\
&=\ii\E[\left< u_0,\nabla u_0\right>-\left< \nabla u_0, u_0\right>]+\ii t\left(\left< Q^{1/2},\nabla Q^{1/2}\right>_{\mathcal{L}_2^0}-\left< \nabla Q^{1/2},Q^{1/2}\right>_{\mathcal{L}_2^0}\right)\\
&=\E[p(u_0)]+\ii t\left(\left< Q^{1/2},\nabla Q^{1/2}\right>_{\mathcal{L}_2^0}-\overline{\left< Q^{1/2},\nabla Q^{1/2}\right>}_{\mathcal{L}_2^0}\right)\\
&=\E[p(u_0)]-2t\mathrm{Im}\left< Q^{1/2},\nabla Q^{1/2}\right>_{\mathcal{L}_2^0}.
\end{align*}
\end{proof}

As we see in the proposition below, the exponential method satisfies the same trace formula for the momentum.
\begin{proposition}
With the same assumptions as in Proposition \ref{prop:momentum}, the exponential integrator \eqref{LinearScheme} exhibits the following formula for the expected momentum.
$$\E[p(u^n)]=\E[p(u_0)]-2t_n\mathrm{Im}\left< Q^{1/2},\nabla Q^{1/2}\right>_{\mathcal{L}_2^0} \quad\text{for all}\quad t_n=nk. $$
\end{proposition}
\begin{proof}
Writing
$$u^n = S(k)u^{n-1}-\ii\int_{t_{n-1}}^{t_n} S(k)\dd W(r),$$
we can use the same techniques as in the proof for the exact solution. We obtain
\begin{align*}
\E[p(u^n)]&=\E[p(u^{n-1})]-2k\mathrm{Im}\left< Q^{1/2},\nabla Q^{1/2}\right>_{\mathcal{L}_2^0}\\
&=\E[p(u_0)]-2t_n \mathrm{Im}\left< Q^{1/2},\nabla Q^{1/2}\right>_{\mathcal{L}_2^0}.
\end{align*}
\end{proof}

\subsection{Numerical experiments for the linear stochastic Schr\"odinger equation}\label{ssect:numexpL}
This subsection illustrates the above properties (error estimates and trace formulas) 
of the stochastic exponential integrator \eqref{LinearScheme} when applied to the linear 
stochastic Schr\"odinger equation \eqref{linSch}. 

Let us first consider the error in the time integration of the linear stochastic Schr\"odinger equation. 
For this, we consider problem \eqref{linSch} on the interval $[0,2\pi]$ with periodic boundary conditions. 
The initial value is taken to be $u_0=0$ and the eigenvalues of the covariance operator $Q$ are 
given by $\lambda_n=1/(1+n^8)$ for $n\in\Z$. Such a regular noise is needed for the midpoint rule to be convergent, 
see \cite[Proposition~3.1]{Weakandstrongorderofconv} and the discussion below. The spatial discretisation is done by a 
pseudospectral method (with $M=2^8$ Fourier modes) using the following representation of the noise 
$$
W(x,t)=\sum_{n\in\Z}\lambda_n^{1/2}\beta_n(t)e_n(x),
$$
where $\{\beta_n(t)\}_{n\in\Z}$ are i.i.d Brownian motion and 
$\{e_n(x)\}_{n\in\Z}=\{\frac{1}{\sqrt{2\pi}}\e^{\ii nx}\}_{n\in\Z}$ is an orthonormal basis of $L^2(0,2\pi)$. 
The resulting system of stochastic differential equations is then integrated in time by 
the stochastic exponential integrator \eqref{LinearScheme} (SEXP), 
the Crank-Nicolson scheme from \cite[Section 3]{Weakandstrongorderofconv}, which reduces to 
the stochastic implicit midpoint rule (MP), and the classical backward Euler-Maruyama methods (BEM). 
Note that the classical explicit Euler-Maruyama would 
require an unreasonable small time step and is therefore omitted in our computational experiments. 
Observe also that MP and BEM are semi-implicit methods whereas SEXP is explicit. 
The rates of mean-square convergence (measured in the $L^2$-norm at the end of the interval 
of integration $[0,0.5]$) of these integrators are presented in Figure~\ref{fig:msLinear}. 
The expected rate of convergence $\bigo{k^2}$ of the stochastic exponential 
integrator, as stated in Theorem~\ref{maxlinear}, can be confirmed. 
Here, the exact solution is approximated by the stochastic midpoint rule with a very small 
time step $k_{\text{exact}}=2^{-10}$ and $M_{\text s}=750000$ samples 
are used for the approximation of the expected values. Note that here and in 
all our numerical experiments, enough samples are taken 
in order for the Monte-Carlo errors to be negligible.

\begin{figure}
\centering
\includegraphics*[height=6.5cm,keepaspectratio]{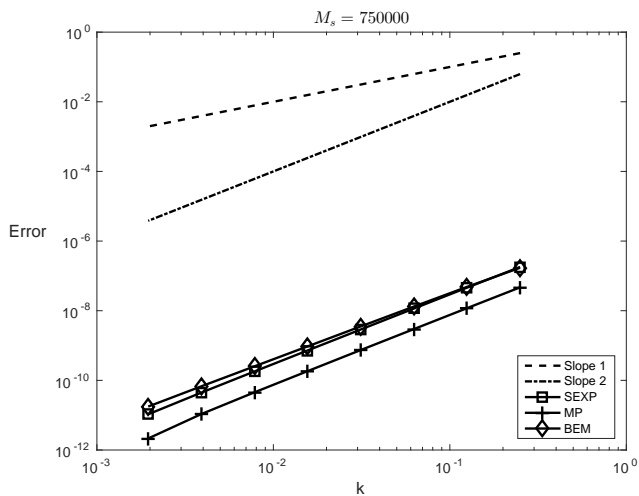}
\caption{Linear stochastic Schr\"odinger equation: 
Mean-square errors for the stochastic exponential integrator (SEXP), 
the midpoint rule (MP), and backward Euler-Maruyama (BEM). 
The dotted, resp. dash-dotted, lines have slopes $1$ and $2$.\label{fig:msLinear}}
\end{figure}

We now turn to the trace formulas. Figure~\ref{fig:traceLin} illustrates the energy trace formula 
from Proposition~\ref{prop:traceL} and compare the results obtained with our 
stochastic exponential method, with the backward Euler-Maruyama method, and with 
the midpoint rule. $M=2^7$ Fourier modes are used for the spatial discretisation; 
$M_{\text s}=10000$ samples are used for the approximatation 
of the expected values; the time interval is $[0,2500]$, 
and all schemes use a step size $k=0.1$. The other parameters are the same as in the 
above numerical experiments. The numerical linear drift in the expected 
value of the energy of the stochastic exponential integrator \eqref{LinearScheme}, 
as stated in Proposition~\ref{traceLEXP}, is observed in this figure. One also 
observes an excellent behaviour of the midpoint rule in reproducing the linear growth of the 
expected energy correctly. This is in contrast with the wrong behaviour 
of the backward Euler-Maruyama method. Note that this fact was already observed 
for the trace formula for the linear stochastic wave equation \cite{cls12s}. 

\begin{figure}
\centering
\includegraphics*[height=6.5cm,keepaspectratio]{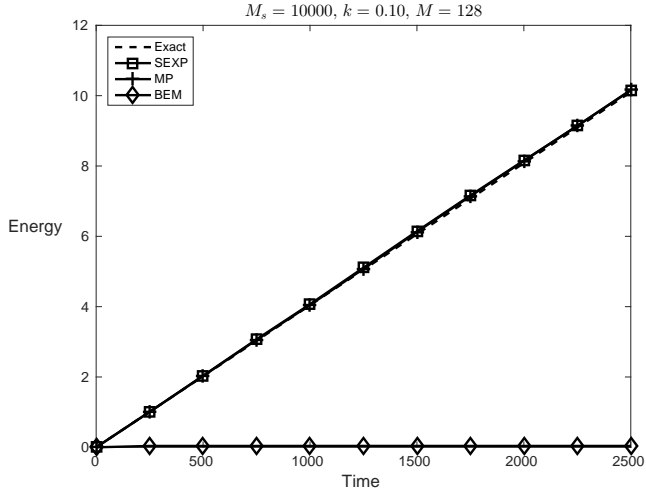}
\caption{Linear stochastic Schr\"odinger equation: 
Energy trace formulas for the numerical solutions given by 
the stochastic exponential method (SEXP), the midpoint rule (MP) and 
the backward Euler-Maruyama scheme (BEM). 
Time interval and time step: $[0,2500]$, resp. $k=0.1$.
\label{fig:traceLin}}
\end{figure}

Using the same parameters as in the previous numerical experiments,  
a similar behaviour of the numerical methods is observed 
for the mass trace formula given in Proposition~\ref{prop:MassL}. These results are not displayed. 
However, when considering a less regular noise, 
for example when the eigenvalues of $Q$ are given by $\lambda_n=1/(1+n^2)$, which is of trace-class,  
one observes that the midpoint rule does not perform as well as the stochastic exponential method, 
see Figure~\ref{fig:masstraceLinear}. This further illustrates the theoretical 
results obtained in Proposition~\ref{prop:MassNum}.

\begin{figure}
\centering
\includegraphics*[height=6.5cm,keepaspectratio]{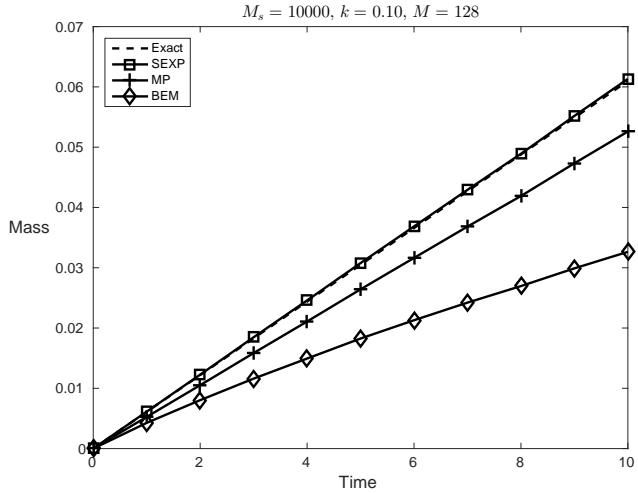}
\caption{Linear stochastic Schr\"odinger equation: 
Mass trace formulas for the numerical solutions given by 
the stochastic exponential method (SEXP), the midpoint rule (MP) and 
the backward Euler-Maruyama scheme (BEM). 
Time interval and time step: $[0,10]$, resp. $k=0.1$. 
\label{fig:masstraceLinear}}
\end{figure}

%


\section{Linear stochastic Schr\"odinger equations with a multiplicative potential}\label{sec-potadd}
We now consider the stochastic Schr\"{o}dinger equation
\begin{equation}
\begin{aligned}
&\ii \mathrm{d}u  = \Delta u\, \mathrm{d}t + V(x)u\,\mathrm{d}t+\,\mathrm{d}W &&\quad \mathrm{in}\: \ 
\mathbb{R}^d\times(0,\infty), \\
&u(\cdot,0) = u_0, &&\quad \mathrm{in}\: \ \mathbb{R}^d,
\end{aligned}
\label{sseMP}
\end{equation}
where $V(x)$ is a real-valued potential. The initial data $u_0$ is an $\mathcal{F}_0$-measurable complex-valued random variable, and the $Q$-Wiener process $\{W(t)\}_{t\geq 0}$ is again complex-valued and square integrable with respect to the filtration. Further assumptions on $d$, $u_0$, $V(x)$ and $Q$ will be made when the results of this section are presented. In this section we aim to prove convergence rate of our exponential integrator, trace formulas for the mass and energy, 
and finally present numerical experiments that confirm these theoretical findings.

We use the same semigroup approach as in Section~\ref{sec-linadd} and write the mild equation 
of \eqref{sseMP} as
\begin{align}\label{mildpotadd}
u(t)=S(t)u_0-\ii\int_0^t S(t-r)V(x)u(r)\,\dd r-\ii\int_0^t S(t-r)\,\dd W(r),
\end{align}
where we recall that $S(t)=\e^{-\ii t\Delta}$.
We have the following result regarding existence, uniqueness and boundedness of solutions of \eqref{sseMP}.
\begin{proposition}\label{potVbdd}
Assume that $\E[\gnorm{u_0}^{2p}]<\infty$ for some $p\in\mathbb{N}$, and $Q^{1/2}\in\mathcal{L}_2^{\sigma}$. Assume also that
\begin{enumerate}[label=(\roman*)]
\item $\sigma>\frac{d}{2}$ and $V(x)\in H^\sigma$, or
\item $\sigma=0$ and $d=1,2,3$ and $V(x)\in H^2$, or
\item $\sigma=1$ and $d=2,3$ and $V(x)\in H^3$.
\end{enumerate}
Then \eqref{sseMP} has a unique solution on $[0,T]$, for $T\in(0,\infty)$, which is given by \eqref{mildpotadd} and satisfies
$$
\E[\sup_{0\leq t\leq T}\gnorm{u(t)}^{2p}]\leq C.
$$
\end{proposition}
\begin{proof}
We first observe that with the different assumptions (i)-(iii), we have
\begin{enumerate}[label=(\roman*)]
\item $H^\sigma$ forms an algebra. In particular we have, for example, $\gnorm{V(x)u(t)}\leq C\gnorm{V(x)}\gnorm{u(t)}$.
\item In this case we have $H^2\subset L^\infty$, so whenever $V(x)\in H^2$ we have $\norm{V(x)u(t)}_{L^2}\leq\norm{V(x)}_{L^\infty}\norm{u(t)}_{L^2}\leq C\norm{u(t)}_{L^2}$.
\item We still have $H^2\subset L^\infty$ and, if $V(x)\in H^3$, then $\norm{V(x)u(r)}_1\leq C\norm{u(r)}_1.$ This can be seen by the following calculation:
\begin{align*}
\norm{V(x)u(r)}_1^2&=\norm{V(x)u(r)}_{L^2}^2+\norm{\nabla(V(x)u(r))}_{L^2}^2\\
&=\norm{V(x)u(r)}_{L^2}^2+\norm{\nabla V(x)\cdot u(r)+V(x)\nabla u(r)}_{L^2}^2\\
&\leq \norm{V(x)u(r)}_{L^2}^2+2(\norm{\nabla V(x)\cdot u(r)}_{L^2}^2+\norm{V(x)\nabla u(r)}_{L^2}^2)\\
&\leq 3\norm{V(x)}_{L^\infty}^2(\norm{u(r)}_{L^2}^2+\norm{\nabla u(r)}_{L^2}^2)+2\norm{\nabla V(x)}_{L^\infty}^2\norm{u(r)}_{L^2}^2\\
&\leq C\norm{u(r)}_1^2,
\end{align*}
since $\norm{\nabla V(x)}_{L^\infty}<\infty$ if $V(x)\in H^3$.
\end{enumerate}
Existence and uniqueness now follow from a standard fixed point argument 
(see \cite[Theorem $7.4$]{DaPrato}) using what was mentioned above about the different cases (i)-(iii).
To prove boundedness, we use the mild equation \eqref{mildpotadd}.
\begin{align*}
\E\left[\sup_{0\leq t\leq T}\gnorm{u(t)}^{2p}\right]&\leq C\left(\E\left[\sup_{0\leq t\leq T}\gnorm{S(t)u_0}^{2p}\right]+\E\left[\sup_{0\leq t\leq T}\gnorm{\int_0^t S(t-r)V(x)u(r)\,\dd r}^{2p}\right]\right.\\
&\quad+\left.\E\left[\sup_{0\leq t\leq T}\gnorm{\int_0^t S(t-r)\,\dd W(r)}^{2p}\right]\right).\\
\end{align*}
We now proceed by estimating the terms on the right-hand side. First we have, by Lemma \ref{lemma1},
$$\E\left[\sup_{0\leq t\leq T}\gnorm{S(t)u_0}^{2p}\right]=\E[\gnorm{u_0}^{2p}]\leq C.$$
For the second term we need to treat each case (i)-(iii) separately.
\begin{enumerate}[label=(\roman*)]
\item Using Lemma \ref{lemma1}, the fact that $H^\sigma$ is an algebra, and H\"{o}lder's inequality, we have
\begin{align*}
\E\left[\sup_{0\leq t\leq T}\gnorm{\int_0^tS(t-r)V(x)u(r)\,\dd r}^{2p}\right]&\leq\E\left[\sup_{0\leq t\leq T}\left(\int_0^t\gnorm{S(t-r)V(x)u(r)}\,\dd r\right)^{2p}\right]\\
&\leq \E\left[\sup_{0\leq t\leq T}\left(\int_0^t\gnorm{V(x)}\gnorm{u(r)}\,\dd r\right)^{2p}\right]\\
&\leq C\gnorm{V(x)}^{2p}\E\left[\sup_{0\leq t\leq T}\left(\int_0^t\gnorm{u(r)}\,\dd r\right)^{2p}\right]\\
&\leq C\E\left[\sup_{0\leq t\leq T}\left(\left(\int_0^t\gnorm{u(r)}^{2p}\,\dd r\right)^{\frac{1}{2p}}\left(\int_0^t 1^{\frac{2p}{2p-1}}\,\dr\right)^{\frac{2p-1}{2p}}\right)^{2p}\right]\\
&\leq CT^{2p-1}\E\left[\sup_{0\leq t\leq T}\int_0^t\gnorm{u(r)}^{2p}\,\dd r\right]\\
&\leq C\int_0^T\E\left[\sup_{0\leq s\leq r}\gnorm{u(s)}^{2p}\right]\,\dd r.
\end{align*}
In the last step we have used that $\gnorm{u(r)}^{2p}\leq\displaystyle\sup_{0\leq s\leq r}\gnorm{u(s)}^{2p}.$

\item If instead $\sigma=0$ and $d=1,2,3$ and thus $H^\sigma=L^2$ is not an algebra, then we have, using $H^2\subset L^\infty$,
\begin{align*}
\E\left[\sup_{0\leq t\leq T}\left(\int_0^t\norm{S(t-r)V(x)u(r)}_{L^2}\,\dd r\right)^{2p}\right]&\leq\norm{V(x)}_{L^\infty}^{2p}\E\left[\sup_{0\leq t\leq T}\left(\int_0^t\norm{u(r)}_{L^2}\,\dd r\right)^{2p}\right]\\
&\leq C\E\left[\sup_{0\leq t\leq T}\left(\int_0^t\norm{u(r)}_{L^2}\,\dd r\right)^{2p}\right].
\end{align*}
The same estimate follows from similar calculations as in (i).

\item We now assume $\sigma=1$, $d=2,3$ and $V(x)\in H^3.$ The proof follows the same lines as in (ii), noting that
$$
\norm{V(x)u(r)}_1\leq C\norm{u(r)}_1,
$$
as seen above.
\end{enumerate}

For the third and last term, using Burkholder's inequality, we obtain
\begin{align*}
\E\left[\sup_{0\leq t\leq T}\gnorm{\int_0^t S(t-r)\,\dd W(r)}^{2p}\right]&\leq C\E\left[\left(\int_0^T\norm{Q^{1/2}}_{\mathcal{L}_2^\sigma}^2\,\dd t\right)^p\right]\\
&\leq CT^p\E[\norm{Q^{1/2}}_{\mathcal{L}_2^\sigma}^{2p}]\leq C.
\end{align*}

Altogether we arrive at
\begin{align*}
\E\left[\sup_{0\leq t\leq T}\gnorm{u(t)}^{2p}\right]\leq C_1+C_2\int_0^T\E[\sup_{0\leq s\leq t}\gnorm{u(s)}^{2p}]\,\dd t,
\end{align*}
and an application of Gronwall's lemma completes the proof.
\end{proof}

\subsection{Error analysis of the stochastic exponential method}
Our exponential scheme for the time integration of \eqref{sseMP} now reads
\begin{align}\label{schemeV}
u^{n}&=S(k)u^{n-1}-\ii kS(k)V(x)u^{n-1}-\ii S(k)\Delta W^{n-1}\nonumber\\
&=S(t_n)u_0-\ii\sum_{j=0}^{n-1}\int_{t_j}^{t_{j+1}}S(t_n-t_j)V(x)u^j\,\dd r-\ii\sum_{j=0}^{n-1}\int_{t_j}^{t_{j+1}}S(t_n-t_j)\,\dd W(r),
\end{align}
where, as before, $k$ denotes the step size, $t_n=nk$ for $n=0,\ldots ,N$, and $\Delta W^{n-1}=W(t_{n})-W(t_{n-1})$. 
We first show that these numerical approximations are bounded.

\begin{proposition}\label{numbound}
Assume that $\E[\gnorm{u_0}^{2p}]<\infty$ for some $p\in\mathbb{N}$, and $Q^{1/2}\in\mathcal{L}_2^\sigma$. As in Proposition \ref{potVbdd}, we also assume
\begin{enumerate}[label=(\roman*)]
\item $\sigma>\frac{d}{2}$, and $V(x)\in H^\sigma,$ or
\item $\sigma=0$ and $d=1,2,3$, and $V(x)\in H^2$, or
\item $\sigma=1$ and $d=2,3$, and $V(x)\in H^3$.
\end{enumerate}
Then the numerial solution given by \eqref{schemeV}, with step size $k$, satisfies
$$
\E[\gnorm{u^n}^{2p}]\leq C\quad\text{for}\quad 0\leq t_n=nk\leq T.
$$
\end{proposition}
\begin{proof}
We prove the case (i), the other cases are treated as in the proof of Proposition~\ref{potVbdd}. We have
\begin{align*}
\E[\gnorm{u^n}^{2p}]&\leq C\left( \E[\gnorm{S(t_n)u_0}^{2p}]+\E\left[\gnorm{\ii k\sum_{j=0}^{n-1}S(t_n-t_j)V(x)u^j}^{2p}\right]\right.\\
&\quad+\left.\E\left[\gnorm{\ii\sum_{j=0}^{n-1}\int_{t_j}^{t_{j+1}}S(t_n-t_j)\,\dd W(r)}^{2p}\right]\right).
\end{align*}
The first term is estimated as in Proposition~\ref{potVbdd}. For the second term, using H\"{o}lder's inequality 
and the fact that $V(x)\in H^\sigma$, we have
\begin{align*}
\gnorm{k\sum_{j=0}^{n-1}S(t_n-t_j)V(x)u^j}&\leq k\sum_{j=0}^{n-1}\gnorm{V(x)u^j}\\
&\leq Ck\sum_{j=0}^{n-1}\gnorm{u^j}\\
&\leq Ck^\frac{1}{2p}\left(\sum_{j=0}^{n-1}\gnorm{u^j}^{2p}\right)^\frac{1}{2p}.
\end{align*}
For the third term we use the Burkholder inequality and obtain
\begin{align*}
\E\left[\gnorm{\sum_{j=0}^{n-1}\int_{t_j}^{t_{j+1}}S(t_n-t_j)\,\dd W(r)}^{2p}\right]&=\E\left[\gnorm{\int_0^{t_n}S(t_n-[r/k]k)\,\dd W(r)}^{2p}\right]\\
&\leq C\E\left[\left(\int_0^T\norm{Q^{1/2}}_{\mathcal{L}_2^\sigma}^2\,\dd t\right)^p\right]\\
&\leq C,
\end{align*}
where, as before, $[r/k]$ denotes the integer part of $r/k$. 

Altogether we arrive at 
$$
\E[\gnorm{u^n}^{2p}]\leq C_1+C_2k\sum_{j=0}^{n-1}\E[\gnorm{u^j}^{2p}]
$$
and an application of Gronwall's lemma concludes the proof.
\end{proof}

We are now ready to prove the main result of this section on the convergence of our numerical method.
\begin{theorem}\label{errorVadd}
Consider the time discretisation of the stochastic Schr\"odinger equation \eqref{sseMP} 
given by the exponential integrator \eqref{schemeV}.   
Let $d=1,2,3$, $\sigma\geq 0$, and assume 
that $\E[\norm{u_0}_{\sigma+2}^{2p}]<\infty$ for some $p\in\mathbb{N}$, $V(x)\in H^{\sigma+2}$, and $Q^{1/2}\in\mathcal{L}_2^{\sigma+2}$. 
Then there exists a constant $C$ such that
$$\E\left[\max_{n=1,\ldots,N}\gnorm{u^n-u(t_n)}^{2p}\right]\leq Ck^{2p}.$$
\end{theorem}

\begin{remark}
Using similar assumptions as in Hypothesis~$4.1$ and~$4.2$ in 
\cite{Weakandstrongorderofconv}, we could replace the term $V(x)u$ in \eqref{sseMP} 
by a general Lipschitz function $F(u)$. That is, if we assume $\sigma>\frac{d}{2}$, 
and $F\colon H^\sigma\to H^\sigma$ is a $C^2$ function that is bounded as well 
as its derivatives up to order $2$, that $(F(u(t))\in L^{2p}(\Omega;L^\infty(0,T; H^{\sigma+2}))$, 
that the solution $(u(t))_{t_\in[0,T]}$ is in $L^{2p}(\Omega;L^\infty(0,T;H^{\sigma+2}))$ 
and $L^{4p}(\Omega,L^\infty(0,T,H^{\sigma+1}))$, and $u_0\in L^{2p}(\Omega; H^{\sigma+2})$, then the convergence order stated 
in Theorem~\ref{errorVadd} remains unchanged.
\end{remark}

\begin{proof}
As before, there are three cases to consider:
\begin{enumerate}[label=(\roman*)]
\item $\sigma>\frac{d}{2}$,
\item $\sigma=0$ and $d=1,2,3$,
\item $\sigma=1$ and $d=2,3$.
\end{enumerate}
We will prove the case (i) where $\sigma>\frac{d}{2}$, so that $H^\sigma$ forms an algebra. 
In the cases (ii) and (iii), $H^\sigma$ does not form an algebra, but these two cases can 
be treated in a similar way as in the proof of Proposition~\ref{potVbdd}. We consider the error 
$$
u^n-u(t_n)=\text{Err}^n_V + \text{Err}^n_W,
$$
where we have defined
$$
\text{Err}^n_V:=\ii\sum_{j=0}^{n-1}\int_{t_j}^{t_{j+1}} (S(t_n-r)V(x)u(r)-S(t_n-t_j)V(x)u^j)\,\dd r,
$$
and
$$
\text{Err}^n_W:=\ii\sum_{j=0}^{n-1}\int_{t_j}^{t_{j+1}}(S(t_n-r)-S(t_n-t_j))\,\dd W(r).
$$
This then gives us 
$$
\E[\max_{n=1,\ldots, N}\gnorm{u^n-u(t_n)}^{2p}]\leq C\left(\E[\max_{n=1,\ldots, N}\gnorm{\text{Err}^n_V}^{2p}]+\E[\max_{n=1,\ldots N}\gnorm{\text{Err}^n_W}^{2p}]\right).
$$
The second term on the right-hand side is the same term as in the proof of Theorem~\ref{maxlinear} and we thus get
$$
\E[\max_{n=1,\ldots, N}\gnorm{\text{Err}^n_W}^{2p}]\leq Ck^{2p}.
$$
In order to estimate $\text{Err}^n_V$, we first write it as
\begin{align*}
\text{Err}^n_V&=i\sum_{j=0}^{n-1}\int_{t_j}^{t_{j+1}} S(t_n-r)V(x)(u(r)-u(t_j))\,\dd r\\
&\quad+i\sum_{j=0}^{n-1}\int_{t_j}^{t_{j+1}}(S(t_n-r)-S(t_n-t_j))V(x)u(t_j)\,\dd r\\
&\quad+i\sum_{j=0}^{n-1}\int_{t_j}^{t_{j+1}}S(t_n-t_j)V(x)(u(t_j)-u^j)\,\dd r\\
&=: I^n_1+I^n_2+I^n_3.
\end{align*}
The estimate for $I^n_1$ is more complicated than the estimates for $I^n_2$ and $I^n_3$, so we save it for last. Using Lemma \ref{lemma1} and that $H^{\sigma+2}$ is an algebra, we have
\begin{align*}
\gnorm{I^n_2}&\leq \sum_{j=0}^{n-1}\int_{t_j}^{t_{j+1}}\gnorm{(S(t_n-r)-S(t_n-t_j))V(x)u(t_j)}\,\dd r\\
&=\sum_{j=0}^{n-1}\int_{t_j}^{t_{j+1}} \gnorm{S(t_n-t_j)(S(t_j-r)-I)V(x)u(t_j)}\,\dd r\\
&\leq\sum_{j=0}^{n-1}\int_{t_j}^{t_{j+1}}|t_j-r|\gnorm{\Delta(V(x)u(t_j))}\,\dd r\\
&\leq Ck^2\sum_{j=0}^{n-1}\norm{V(x)u(t_j)}_{\sigma+2}\\
&\leq Ck^2\norm{V(x)}_{\sigma+2}\sum_{j=0}^{n-1} \norm{u(t_j)}_{\sigma+2}.
\end{align*}
Now, using H\"{o}lder's inequality we have
\begin{align*}
\sum_{j=0}^{n-1} \norm{u(t_j)}_{\sigma+2}&\leq \left(\sum_{j=0}^{n-1}\norm{u(t_j)}_{\sigma+2}^{2p}\right)^{\frac{1}{2p}}k^{\frac{1-2p}{2p}}\\
&\leq Ck^{-1}\left(\sup_{0\leq t\leq T}\norm{u(t)}_{\sigma+2}^{2p}\right)^{\frac{1}{2p}}.
\end{align*}
Thus
$$\E[\max_{n=1,\ldots,N}\gnorm{I^n_2}^{2p}]\leq Ck^{2p}\E[\sup_{0\leq t\leq T}\norm{u(t)}_{\sigma+2}^{2p}]\leq Ck^{2p},$$
by Proposition \ref{potVbdd}.

Again using the fact that $H^\sigma$ is an algebra, and then H\"{o}lder's inequality, we have
\begin{align*}
\gnorm{I^n_3}&\leq \sum_{j=0}^{n-1}\int_{t_j}^{t_{j+1}}\gnorm{V(x)(u(t_j)-u^j)}\,\dd r\\
&\leq C\gnorm{V(x)}\sum_{j=0}^{n-1}\int_{t_j}^{t_{j+1}} \gnorm{u(t_j)-u^j}\,\dd r\\
&\leq Ck\sum_{j=0}^{n-1}\gnorm{u(t_j)-u^j}\\
&\leq Ck\left(\sum_{j=0}^{n-1}\gnorm{u(t_j)-u^j}^{2p}\right)^\frac{1}{2p}k^\frac{1-2p}{2p}\\
&\leq Ck^\frac{1}{2p}\left(\sum_{j=0}^{n-1}\gnorm{u(t_j)-u^j}^{2p}\right)^\frac{1}{2p}.
\end{align*}
Thus
$$\E[\max_{n=1,\ldots,N}\gnorm{I^n_3}^{2p}]\leq Ck\sum_{j=0}^{N-1}\E[\max_{l=0,\ldots,j}\gnorm{u(t_l)-u^l}^{2p}].$$

In order to estimate $I^n_1$ we use the mild formulation of the exact solution on the interval $[t_j,r]$ 
$$u(r)=S(r-t_j)u(t_j)-\ii\int_{t_j}^r S(r-\rho)V(x)u(\rho)\,\dd\rho-\ii\int_{t_j}^r S(r-\rho)\,\dd W(\rho),$$
and therefore
$$u(r)-u(t_j)=(S(r-t_j)-I)u(t_j)-\ii\int_{t_j}^r S(r-\rho)V(x)u(\rho)\,\dd\rho-\ii\int_{t_j}^r S(r-\rho)\,\dd W(\rho).$$
We have
\begin{align*}
I^n_1&=\ii\sum_{j=0}^{n-1}\int_{t_j}^{t_{j+1}} S(t_n-r)V(x)(u(r)-u(t_j))\,\dd r\\
&=i\sum_{j=0}^{n-1}\int_{t_j}^{t_{j+1}} S(t_n-r)V(x)(S(r-t_j)-I)u(t_j)\,\dd r\\
&\quad+\sum_{j=0}^{n-1}\int_{t_j}^{t_{j+1}} S(t_n-r)V(x)\left(\int_{t_j}^r S(r-\rho)V(x)u(\rho)\,\dd\rho\right)\,\dd r\\
&\quad+\sum_{j=0}^{n-1}\int_{t_j}^{t_{j+1}} S(t_n-r)V(x)\left(\int_{t_j}^r S(r-\rho)\,\dd W(\rho)\right)\,\dd r\\
&=: J^n_1+J^n_2+J^n_3.
\end{align*}
Using Lemma~\ref{lemma1}, the fact that $H^\sigma$ is an algebra, H\"{o}lder's inequality, 
and Proposition \ref{potVbdd}, we obtain
\begin{align*}
\gnorm{J^n_1}&\leq\sum_{j=0}^{n-1}\int_{t_j}^{t_{j+1}}\gnorm{S(t_n-r)V(x)(S(r-t_j)-I)u(t_j)}\,\dd r\\
&\leq C\gnorm{V(x)}\sum_{j=0}^{n-1}\int_{t_j}^{t_{j+1}}|r-t_j|\norm{u(t_j)}_{\sigma+2}\,\dd r\\
&\leq Ck^2\sum_{j=0}^{n-1}\norm{u(t_j)}_{\sigma+2}\\
&\leq Ck\left(\sup_{0\leq t\leq T}\norm{u(t)}_{\sigma+2}^{2p}\right)^{\frac{1}{2p}}\\
&\leq Ck,
\end{align*}
so that
$$
\E[\max_{n=1,\ldots,N}\gnorm{J^n_1}^{2p}]\leq Ck^{2p}.
$$

Next, we estimate $J^n_2$. Using Lemma~\ref{lemma1}, the fact that $H^\sigma$ is an algebra, 
and H\"{o}lder's inequality, we get
\begin{align*}
\gnorm{J^n_2}&\leq \sum_{j=0}^{n-1}\int_{t_j}^{t_{j+1}}\gnorm{S(t_n-r)V(x)\left(\int_{t_j}^r S(r-\rho)V(x)u(\rho)\,\dd\rho\right)}\,\dd r\\
&\leq C\gnorm{V(x)}\sum_{j=0}^{n-1}\int_{t_j}^{t_{j+1}}\int_{t_j}^r\gnorm{V(x)}\gnorm{u(\rho)}\,\dd\rho\,\dd r\\
&\leq C\sum_{j=0}^{n-1}\int_{t_j}^{t_{j+1}}|r-t_j|^{\frac{2p-1}{2p}}\left(\int_{t_j}^r \gnorm{u(\rho)}^{2p}\,\dd\rho\right)^\frac{1}{2p}\,\dd r\\
&\leq C\sum_{j=0}^{n-1}\int_{t_j}^{t_{j+1}}|r-t_j|\left(\sup_{t_j\leq\rho\leq r}\gnorm{u(\rho)}^{2p}\right)^\frac{1}{2p}\,\dd r\\
&\leq Ck\left(\sup_{0\leq t\leq t_n}\gnorm{u(t)}^{2p}\right)^\frac{1}{2p},
\end{align*}
and thus, by Proposition \ref{potVbdd},
$$\E[\max_{n=1,\ldots,N}\gnorm{J^n_2}^{2p}]\leq Ck^{2p}\E[\sup_{0\leq t\leq T}\gnorm{u(t)}^{2p}]\leq Ck^{2p}.$$

Finally, using Fubini's theorem (when changing the order of integration, 
we go from $t_j\leq r\leq t_{j+1}$ and $t_j\leq\rho\leq r$ to 
$t_j\leq\rho\leq t_{j+1}$ and $\rho\leq r\leq t_{j+1}$), 
the Burkholder inequality, Lemma~\ref{lemma1}, and the fact that $H^\sigma$ is an algebra, we have 
\begin{align*}
\E[\max_{n=1,\ldots,N}\gnorm{J^n_3}^{2p}]=&\E\left[\max_{n=1,\ldots,N}\gnorm{\sum_{j=0}^{n-1}\int_{t_j}^{t_{j+1}} S(t_n-r)V(x)\left(\int_{t_j}^r S(r-\rho)\,\dd W(\rho)\right)\,\dd r}^{2p}\right]\\
&=\E\left[\max_{n=1,\ldots,N}\gnorm{\sum_{j=0}^{n-1}\int_{t_j}^{t_{j+1}}\int_\rho^{t_{j+1}}S(t_n-r)V(x)S(r-\rho)\,\dd r\,\dd W(\rho)}^{2p}\right]\\
&=\E\left[\max_{n=1,\ldots,N}\gnorm{\int_0^{t_n}\int_\rho^{[\frac{\rho}{k}+1]k}S(t_n-r)V(x)S(r-\rho)\,\dd r\,\dd W(\rho)}^{2p}\right]\\
&\leq \E\left[\sup_{0\leq t\leq T}\gnorm{\int_0^t\int_\rho^{[\frac{\rho}{k}+1]k}S(t-r)V(x)S(r-\rho)\,\dd r\,\dd W(\rho)}^{2p}\right]\\
&\leq C\sup_{0\leq t\leq T}\E\left[\norm{\int_0^t\int_\rho^{[\frac{\rho}{k}+1]k}S(t-r)V(x)S(r-\rho)\,\dd r\,\dd W(\rho)}_\sigma^{2p}\right]\\
&\leq C\sup_{0\leq t\leq T}\E\left[\left(\int_0^t\norm{\int_\rho^{[\frac{\rho}{k}+1]k}S(t-r)V(x)S(r-\rho)Q^{1/2}\,\dd r}_{\mathcal{L}_2^\sigma}^2\,\dd \rho\right)^p\right]\\
&=C\max_{n=1,\ldots,N}\E\left[\left(\sum_{j=0}^{n-1}\int_{t_j}^{t_{j+1}}\norm{\int_\rho^{t_{j+1}}S(t-r)V(x)S(r-\rho)Q^{1/2}\,\dd r}_{\mathcal{L}_2^\sigma}^2\,\dd \rho\right)^p\right]\\
&\leq C\max_{n=1,\ldots,N}\E\left[\left(\sum_{j=0}^{n-1}\int_{t_j}^{t_{j+1}}\left(\int_\rho^{t_{j+1}}\gnorm{V(x)}\norm{Q^{1/2}}_{\mathcal{L}_2^\sigma}\,\dd r\right)^2\,\dd \rho\right)^p\right]\\
&\leq C\E\left[\left(\sum_{j=0}^{N-1}\int_{t_j}^{t_{j+1}}|t_{j+1}-\rho|^2\,\dd \rho\right)^p\right]\\
&\leq Ck^{2p},
\end{align*}
where we recall that $[\frac{\rho}{k}+1]$ denotes the integer part 
of $\frac{\rho}{k}+1$, where $k$ is the step size of the scheme.

To summarize, we have obtained the following bound  
$$
\E[\max_{n=1,\ldots,N}\gnorm{u^n-u(t_n)}^{2p}]\leq C_1k^{2p}+C_2k\sum_{j=0}^{N-1}\E[\max_{l=0,\ldots,j}\gnorm{u^l-u(t_l)}^{2p}],
$$
and an application of Gronwall's lemma completes the proof.
\end{proof}

\subsection{A trace formula for the mass}
With appropriate boundary conditions, such as when the domain is the $d$-dimensional torus $\mathbb{T}^d$, the expected mass
$$
\E[M(u)]=\E\left[\int_{\mathbb{T}^d}|u|^2\,\dd x\right]
$$
still exhibits linear growth. On a smooth compact Riemannian manifold $(\widetilde{M},g)$, such as the $d$-dimensional torus, we still have $H^2(\widetilde{M})\subset L^\infty(\widetilde{M})$ if $d=1,2,3$ (see \cite[Theorem 2.7]{Hebey}). In particular, Propositions \ref{potVbdd} and \ref{numbound} and Theorem \ref{errorVadd} above remain valid when $\mathbb{R}^d$ is replaced by $\mathbb{T}^d$. Therefore, in this subsection, we assume that the spatial domain is the $d$-dimensional torus $\mathbb{T}^d$.

\begin{proposition}\label{MassExact}
Let $d=1,2,3$ and assume that the initial data is such that $\E[M(u_0)]$ is finite, that $V(x)\in H^2$, and that $Q$ is trace-class. Then the exact solution of \eqref{sseMP}, given by equation \eqref{mildpotadd}, satisfies the trace formula
$$
\E[M(u(t))]=\E[M(u_0)]+t\Tr(Q),
$$
for $t\geq 0$.
\end{proposition}
\begin{proof}
This follows from the Ito formula \cite[Theorem $4.17$]{DaPrato}. We have
\begin{align*}
M(u(t))&=M(u_0)+\int_0^t\left( M'(u(s)),\Phi(s)\,\dd W(s)\right)+\int_0^t \left( M'(u(s)),\varphi(s)\right)\,\dd s\\
&\quad+\int_0^t\frac{1}{2}\Tr[M''(u(s))(\Phi(s)Q^{1/2})(\Phi(s)Q^{1/2})^*]\,\dd s,
\end{align*}
where $\varphi(s)=-\ii\Delta u(s)-\ii V(x)u(s)$ and $\Phi(s)=I$. 
The expected value of the first integral on the right-hand side is zero. For the second integral, we have
\begin{align*}
\left( M'(u(s)),\varphi(s)\right)&=(u(s),\varphi(s))+(\varphi(s),u(s))\\
&=\left(u(s),-\ii\Delta u(s)\right)+(u(s),-\ii V(x)u(s))\\
&\quad+(-\ii\Delta u(s),u(s))+(-\ii V(x)u(s),u(s))\\
&=0.
\end{align*}
Thus, we arrive at
\begin{align*}
\E[M(u(t))]&=\E[M(u_0)]+\frac{1}{2}\int_0^t \E[\Tr(M''(u(s))(Q^{1/2})(Q^{1/2})^*)]\,\dd s\\
&=\E[M(u_0)]+t\Tr(Q).
\end{align*}
\end{proof}

We now investigate how the expected mass of our exponential integrator behaves. 
As seen below we do not get an exact trace formula as for the linear problem considered in Section~\ref{sec-linadd}, 
but instead we get a trace formula with an error of size $\mathcal{O}(k)$ 
on compact intervals.
%
\begin{proposition}\label{prop-traceP} 
Let $d=1,2,3$ and assume that the initial data is such that $\E[M(u_0)]$ is finite, that $V(x)\in H^2$, and that $Q$ is trace-class. Then the stochastic exponential integrator $u^n$ given by \eqref{schemeV} satisfies the following almost trace formula for the expected mass
$$
\E[M(u^n)]=\E[M(u_0)]+t_n\Tr(Q)+\mathcal{O}(k)\quad\text{for}\quad0\leq t_n=nk\leq T.
$$
\end{proposition}
\begin{proof}
First note that, by assumption, we have $\E[\norm{u_0}_{L^2}^2]<\infty$.

The numerical method reads
$$
u^n=S(k)u^{n-1}-\ii S(k)V(x)u^{n-1}k-\ii\int_{t_{n-1}}^{t_n}S(k)\,\dd W(r)
$$
and a straight-forward calculation gives
\begin{align*}
\norm{u^n}_{L^2}^2 &= \norm{S(k)u^{n-1}}_{L^2}^2+(S(k)u^{n-1},-\ii S(k)V(x)u^{n-1}k)\\
&\quad+\left(S(k)u^{n-1},-\ii\int_{t_{n-1}}^{t_n}S(k)\,\dd W(r)\right)+(-\ii S(k)V(x)u^{n-1}k,S(k)u^{n-1})\\
&\quad+\norm{\ii S(k)V(x)u^{n-1}k}_{L^2}^2+\left(-\ii S(k)V(x)u^{n-1}k,-\ii\int_{t_{n-1}}^{t_n}S(k)\,\dd W(r)\right)\\
&\quad+\left(-\ii\int_{t_{n-1}}^{t_n}S(k)\,\dd W(r),S(k)u^{n-1}\right)+\left(-\ii\int_{t_{n-1}}^{t_n}S(k)\,\dd W(r),-\ii S(k)V(x)u^{n-1}k\right)\\
&\quad+\norm{\ii\int_{t_{n-1}}^{t_n}S(k)\,\dd W(r)}_{L^2}^2.
\end{align*}
Next we observe that $(S(k)u^{n-1},-\ii S(k)V(x)u^{n-1}k)+(-\ii S(k)V(x)u^{n-1}k,S(k)u^{n-1})=0$ 
and the expected value of the terms containing one stochastic integral is zero. We thus have, using Lemma~\ref{lemma1}, 
Proposition~\ref{numbound} (case (ii)), and Ito's isometry,
\begin{align*}
\E[\norm{u^n}_{L^2}^2] &= \E[\norm{u^{n-1}}_{L^2}^2]+k^2\E[\norm{V(x)u^{n-1}}_{L^2}^2]+\E\left[\norm{\int_{t_{n-1}}^{t_n}S(k)\,\dd W(r)}_{L^2}^2\right]\\
&=\E[M(u^{n-1})]+Ck^2+\int_{t_{n-1}}^{t_n}\norm{Q^{1/2}}_{\mathcal{L}_2^0}^2\,\dd r\\
&=\E[M(u^{n-1})]+Ck^2+k\Tr(Q).
\end{align*}
An iteration completes the proof.
\end{proof}

\subsection{A trace formula for the energy}
As for the mass, with suitable boundary conditions, the expected energy will still grow linearly. 
We again consider the domain to be the $d$-dimensional torus and the energy is given by
$$
H(u)=\frac{1}{2}\int_{\mathbb{T}^d}|\nabla u|^2\,\dd x-\frac{1}{2}\int_{\mathbb{T}^d}V(x)|u|^2\,\dd x.
$$

\begin{proposition}
Let $d=1,2,3$ and assume that the initial data is such that $\E[H(u_0)]$ is finite, 
and that $Q^{1/2}\in\mathcal{L}_2^1$. If $d=1$, we assume $V(x)\in H^2$ and 
if $d=2,3$ we assume $V(x)\in H^3$. Then the exact solution of \eqref{sseMP}, 
given by equation \eqref{mildpotadd}, satisfies the trace formula
$$\E[H(u(t))]=\E[H(u_0)]+\frac{t}{2}(\Tr(\nabla Q\nabla)-\Tr(Q^{1/2}V(x)Q^{1/2})).$$
\end{proposition}

\begin{proof}
Similar to the trace formula for the mass, we use the Ito formula
\begin{align*}
H(u(t))&=H(u_0)+\int_0^t\left( H'(u(s)),\Phi(s)\,\dd W(s)\right)+\int_0^t \left( H'(u(s)),\varphi(s)\right)\,\dd s\\
&\quad+\int_0^t\frac{1}{2}\Tr[H''(u(s))(\Phi(s)Q^{1/2})(\Phi(s)Q^{1/2})^*]\,\dd s,
\end{align*}
where $\varphi(s)=-\ii\Delta u(s)-\ii V(x)u(s)$ and $\Phi(s)=I$. 
The expected value of the first integral is seen to be zero, 
and a similar calculation as in the proof of Proposition~\ref{MassExact} 
shows that the second integral is zero as well. To calculate the third integral, we note that
\begin{align*}
\Tr(H''(u(s))(Q^{1/2})(Q^{1/2})^*)&=\sum_{n\geq 1}H''(u(s))(Q^{1/2}e_n,Q^{1/2}e_n)\\
&=\sum_{n\geq 1}(\nabla Q^{1/2}e_n,\nabla Q^{1/2}e_n)_{L^2}-\sum_{n\geq 1}(V(x)Q^{1/2}e_n,Q^{1/2}e_n)_{L^2}\\
&=\Tr(\nabla Q\nabla)-\Tr(Q^{1/2}V(x)Q^{1/2}).
\end{align*}
\end{proof}

The energy for the numerical approximation satisfies an almost trace formula: 
we get a small error of size $\mathcal{O}(k)$, as seen in the next proposition.
\begin{proposition} 
Let $d=1,2,3$ and assume that $\E[\norm{u_0}_2^2]\leq C.$ 
If $d=1$ then assume $V(x)\in H^2$. If $d=2,3$ then assume $V(x)\in H^3$. Also assume that $Q^{1/2}\in\mathcal{L}_2^2$. Then the stochastic exponential integrator $u^n$ given by \eqref{LinearScheme} satisfies the following almost trace formula for the expected energy
\begin{align*}
\E[H(u^{n})]&=\E[H(u_0)]+\frac{t_n}{2}(\Tr(\nabla Q\nabla)-\Tr(Q^{1/2}V(x)Q^{1/2}))+\mathcal{O}(k)\quad\text{for}\quad0\leq t_n=nk\leq T.
\end{align*}
\end{proposition}
\begin{proof}
First we note that with these assumptions we have finite initial energy, since
\begin{align*}
\E[H(u_0)]&=\E[\norm{\nabla u_0}_{L^2}^2]-\E\left[\int_{\mathbb{T}^d}V(x)|u_0|^2\,\dd x\right]\\
&\leq \E[\norm{u_0}_1^2]+\norm{V(x)}_{L^\infty}^2\E[\norm{u_0}_{L^2}^2]\\
&\leq C.
\end{align*}
We add and subtract the energy for the exact solution
\begin{align*}
\E[H(u^n)]&=\E[H(u(t_n))]+\E[H(u^n)-H(u(t_n))]\\
&=\E[H(u_0)]+\frac{t_n}{2}(\Tr(\nabla Q\nabla)-\Tr(Q^{1/2}V(x)Q^{1/2}))+\E[H(u^n)-H(u(t_n))].
\end{align*}
Thus we need to show that $|\E[H(u^n)-H(u(t_n))]|\leq Ck$. We have (omitting the constant $1/2$ for ease of presentation)
\begin{align*}
|\E[H(u^n)-H(u(t_n))]|&\leq |\E[\norm{\nabla u^n}_{L^2}^2-\norm{\nabla u(t_n)}_{L^2}^2]|\\
&\quad+\left|\E\left[\int_{\mathbb{T}^d}V(x)|u(t_n)|^2\,\dd x-\int_{\mathbb{T}^d}V(x)|u^n|^2\,\dd x\right]\right|\\
&=:I^n_1+I^n_2.
\end{align*}
We begin by estimating $I^n_1$. We have, using that the inner product defined by \eqref{innerp} is symmetric, 
integration by parts and the boundary conditions, Cauchy-Schwarz and H\"{o}lder's inequalities,
\begin{align*}
I^n_1 &=|\E[(\nabla(u^n+u(t_n)),\nabla(u^n-u(t_n)))]|\\
&=|\E[(\Delta(u^n+u(t_n)),u^n-u(t_n))]|\\
&\leq \E[\norm{\Delta(u^n+u(t_n))}_{L^2}\norm{u^n-u(t_n)}_{L^2}]\\
&\leq \left(\E[\norm{\Delta(u^n+u(t_n))}_{L^2}^2]\right)^{1/2}\left(\E[\norm{u^n-u(t_n)}_{L^2}^2]\right)^{1/2}\\
&\leq \left(\E[\norm{u^n+u(t_n)}_2^2]\right)^{1/2}\left(\E[\norm{u^n-u(t_n)}_{L^2}^2]\right)^{1/2}.
\end{align*}
Using Propositions~\ref{potVbdd},~\ref{numbound}, and Theorem~\ref{errorVadd} we have
$$I^n_1\leq Ck.$$
Similarly, we have
\begin{align*}
I^n_2&=\left|\E\left[\int_{\mathbb{T}^d} V(x)(|u(t_n)|^2-|u^n|^2)\,\dd x\right]\right|\\
&\leq \norm{V(x)}_{L^\infty}|\E[\norm{u(t_n)}_{L^2}^2-\norm{u^n}_{L^2}^2]|\\
&\leq C\E[\norm{u(t_n)+u^n}_{L^2}\norm{u(t_n)-u^n}_{L^2}]\\
&\leq C\left(\E[\norm{u(t_n)+u^n}_{L^2}^2]\right)^{1/2}\left(\E[\norm{u(t_n)-u^n}_{L^2}^2]\right)^{1/2}\\
&\leq Ck.
\end{align*}
\end{proof}

\subsection{Numerical experiments for the stochastic Schr\"odinger equations with a multiplicative potential}
This subsection illustrates some of the above properties (error estimates and trace formula for the mass) 
of the stochastic exponential integrator \eqref{schemeV} when applied to the linear 
stochastic Schr\"odinger equation with a multiplicative potential on the interval $[0,2\pi]$ 
with periodic boundary conditions \eqref{sseMP}. In these numerical 
experiments, we consider a potential $V(x)=\frac{1}{1+\sin^2(x)}$, \cite{bos06}, and initial values with 
$u_0(x)=\frac{2}{2-\cos(x)}$. We compare the stochastic exponential integrator (SEXP) 
with the Crank-Nicolson scheme (CN) and the semi-implicit Euler-Maruyama scheme, 
explicit in $V(x)u$, (SEM). 

Figure~\ref{fig:msPotenial} illustrates the convergence errors of the above numerical methods 
for a noise with covariance operator having the eigenvalues $\lambda_n=1/(1+n^6)$ for $n\in\Z$. 
The spatial discretisation is done by a pseudospectral method with $M=2^8$ Fourier modes. 
The rates of mean-square convergence (measured in the $L^2$-norm at the end of the interval 
of integration $[0,0.5]$) of these numerical methods are presented in this figure. 
The expected rate of convergence $\bigo{k^2}$ of the stochastic exponential 
integrator, as stated in Theorem~\ref{errorVadd}, can be confirmed. 
Here, the exact solution is approximated by the stochastic exponential method with a very small 
time step $k_{\text{exact}}=2^{-9}$ and $M_{\text s}=750000$ samples 
are used for the approximation of the expected values.

\begin{figure}
\centering
\includegraphics*[height=6.5cm,keepaspectratio]{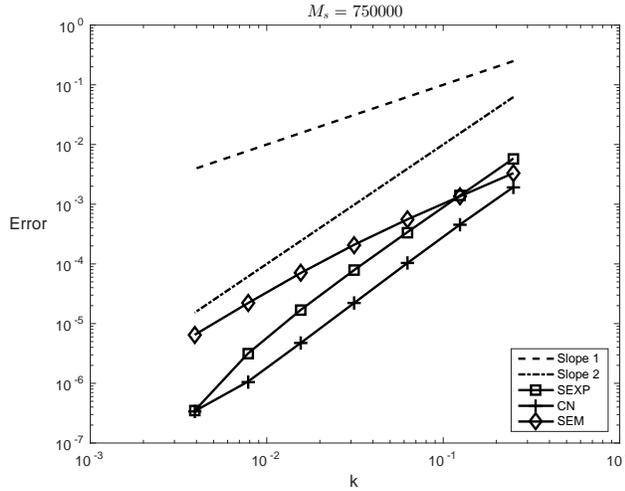}
\caption{Stochastic Schr\"odinger equation with potential: 
Mean-square errors for the stochastic exponential integrator (SEXP), 
the Crank-Nicolson scheme (CN), and the semi-implicit Euler-Maruyama (SEM). 
The dotted, resp. dash-dotted, lines have slopes $1$ and $2$. 
\label{fig:msPotenial}}
\end{figure}

We now consider eigenvalues of $Q$ given by $\lambda_n=1/(1+n^2)$ and examine the 
numerical trace formulas for the mass. Figure~\ref{fig:masstracePotential} is produced using 
$M=2^7$ Fourier modes for the spatial discretisation; $M_{\text s}=10000$ 
samples for the approximatation of the expected values; 
the time interval is $[0,5]$; and a step size $k=0.1$. 
As stated by Proposition~\ref{prop-traceP}, in this figure, one can observe the small error 
in the preservation of the trace formula for the mass 
of the numerical solution given by the exponential integrator.  

\begin{figure}
\centering
\includegraphics*[height=6.5cm,keepaspectratio]{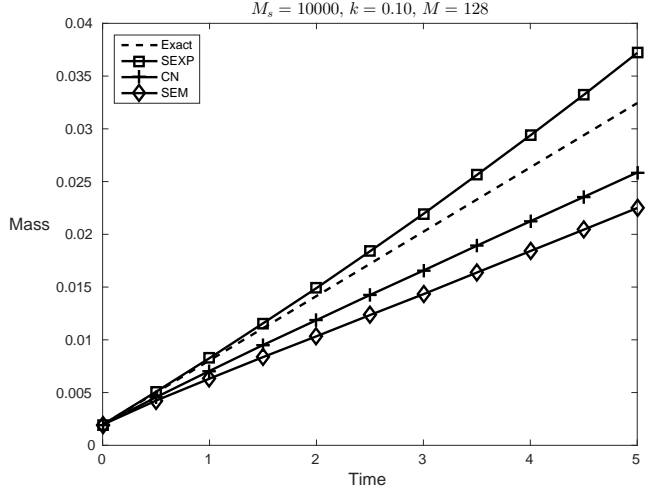}
\caption{Stochastic Schr\"odinger equation with potential: 
Mass trace formulas for the numerical solutions given by 
the stochastic exponential method (SEXP), the Crank-Nicolson scheme (CN) and 
the semi-implicit Euler-Maruyama scheme (SEM). 
Time interval and time step: $[0,5]$, resp. $k=0.1$. 
\label{fig:masstracePotential}}
\end{figure}


\section{Stochastic Schr\"odinger equations driven by multiplicative noise}\label{sec-mult}

We continue our exposition of the exponential integrator applied to the stochastic Schr\"{o}dinger equation 
by considering the equation driven by a multiplicative Ito noise
\begin{equation}
\begin{aligned}
&\ii \mathrm{d}u  = \Delta u\, \mathrm{d}t + V(x)u\,\mathrm{d}t+u\,\mathrm{d}W &&\quad \mathrm{in}\: \ 
\mathbb{R}^d\times(0,\infty), \\
&u(\cdot,0) = u_0 &&\quad \mathrm{in}\: \ \mathbb{R}^d.
\end{aligned}
\label{sse-mult}
\end{equation}
As before, $V(x)$ is a real-valued potential, $u_0$ is an $\mathcal{F}_0$-measurable complex-valued random variable, and the $Q$-Wiener process is complex-valued and square integrable. Further assumptions on $d$, $u_0$, $V(x)$, and $Q$ will be made in the results of this section.
Again, the equation can be written in its mild form
\begin{align*}
u(t)=S(t)u_0-\ii\int_0^t S(t-r)V(x)u(r)\,\dd r-\ii\int_0^t S(t-r)u(r)\,\dd W(r),
\end{align*}
where $S(t)=\e^{-\ii t\Delta}$.

\begin{proposition}\label{multWbdd}
Let $\sigma>\frac{d}{2}$, and assume that $\E[\gnorm{u_0}^{2p}]<\infty$ for some $p\in\mathbb{N}$, $V(x)\in H^\sigma$, and $Q^{1/2}\in\mathcal{L}_2^\sigma$. Then there exists a unique solution $u(t)$ on $[0,T]$, for some $T>0$, to problem \ref{sse-mult} which satisfies
$$\E[\sup_{0\leq t\leq T}\gnorm{u(t)}^{2p}]\leq C.$$
\end{proposition}
\begin{remark}\label{remark1}
The result of Proposition~\ref{multWbdd} holds with $\sigma=0$ if 
we assume $d=1,2,3,$ $\E[\norm{u_0}_0^{2p}]<\infty$ for some $p\in\mathbb{N}$, 
$V(x)\in H^2$, and $Q^{1/2}\in\mathcal{L}_2^2$. 
\end{remark}

\begin{proof}
The existence and uniqueness of a solution again follows by a standard fixed point argument as used in \cite[Theorem $7.4$]{DaPrato}.
From the mild equation and the proof of Proposition \ref{potVbdd}, we have
\begin{align*}
\E\left[\sup_{0\leq t\leq T}\gnorm{u(t)}^{2p}\right]&\leq C\left(\E[\sup_{0\leq t\leq T}\gnorm{S(t)u_0}^{2p}]\right.\\
&\quad+ \E\left[\sup_{0\leq t\leq T}\gnorm{\int_0^t S(t-r)V(x)u(r)\,\dd r}^{2p}\right]\\
&\quad+\left.\E\left[\sup_{0\leq t\leq T}\gnorm{\int_0^t S(t-r)u(r)\,\dd W(r)}^{2p}\right]\right)\\
&\leq C_1+C_2\int_0^T\E[\sup_{0\leq s\leq t}\gnorm{u(s)}^{2p}]\,\dd t\\
&\quad+C_3\E\left[\sup_{0\leq t\leq T}\gnorm{\int_0^t S(t-r)u(r)\,\dd W(r)}^{2p}\right].
\end{align*}
We need to bound the stochastic integral. Using the Burkholder inequality, that $H^\sigma$ forms an algebra, and H\"{o}lder's inequality, we have
\begin{align*}
\E\left[\sup_{0\leq t\leq T}\gnorm{\int_0^t S(t-r)u(r)\,\dd W(r)}^{2p}\right]&\leq C\E\left[\left(\int_0^T\norm{u(t)Q^{1/2}}_{\mathcal{L}_2^\sigma}^2\,\dd t\right)^p\right]\\
&\leq C\E\left[\left(\int_0^T\gnorm{u(t)}^2\norm{Q^{1/2}}_{\mathcal{L}_2^\sigma}^2\,\dd t\right)^p\right]\\
&\leq C\E\left[\int_0^T\gnorm{u(t)}^{2p}\,\dd t\right]\\
&\leq C\int_0^T\E[\sup_{0\leq s\leq t}\gnorm{u(s)}^{2p}]\,\dd t.
\end{align*}
The result follows from Gronwall's lemma.
\end{proof}

We will need the following lemma regarding regularity of the exact solution.
\begin{lemma}\label{reg}
Let $\sigma>\frac{d}{2}$ and assume $\E[\norm{u_0}_{\sigma+1}^{2p}]<\infty$ for some $p\in\mathbb{N}$, $V(x)\in H^{\sigma+1}$, 
and $Q^{1/2}\in\mathcal{L}_2^{\sigma+1}$. Then, for $0\leq s\leq t<\infty,$ the exact solution of \eqref{sse-mult} satisfies the following regularity estimate
$$\E[\gnorm{u(t)-u(s)}^{2p}]\leq C|t-s|^p.$$
\end{lemma}
\begin{remark}\label{remark2}
The result of Lemma~\ref{reg} holds with $\sigma=0$ if we assume 
$d=1,2,3,$ $\E[\norm{u_0}_2^{2p}]<\infty$ for some $p\in\mathbb{N}$, $V(x)\in H^2,$ 
and $Q^{1/2}\in\mathcal{L}_2^2$.
\end{remark}

\begin{proof}
Using the mild equation, we have
\begin{align*}
u(t)-u(s)&=(S(t-s)-I)u(s)-\ii\int_{s}^{t}S(t-r)V(x)u(r)\,\dd r\\
&\quad-\ii\int_s^t S(t-r)u(r)\,\dd W(r).
\end{align*}
We proceed to estimate the three terms on the right hand side. The first one, due to Lemma \ref{lemma2}, reads
\begin{align*}
\gnorm{(S(t-s)-I)u(s)}&\leq \norm{S(t-s)-I}_{\mathcal{L}(H^{\sigma+1},H^\sigma)}\norm{u(s)}_{\sigma+1}\\
&\leq C|t-s|^{1/2}\norm{u(s)}_{\sigma+1},
\end{align*}
so that, by Proposition \ref{multWbdd},
$$\E[\gnorm{(S(t-s)-I)u(s)}^{2p}]\leq C|t-s|^p.$$
To estimate the second term, we use that $H^\sigma$ forms an algebra and H\"{o}lder's inequality
\begin{align*}
\gnorm{\int_s^t S(t-r)V(x)u(r)\,\dd r}&\leq \int_s^t\gnorm{V(x)u(r)}\,\dd r\\
&\leq C\int_{s}^t\gnorm{V(x)}\gnorm{u(r)}\,\dd r\\
&\leq C|t-s|\left(\sup_{s\leq r\leq t}\gnorm{u(r)}^{2p}\right)^\frac{1}{2p}.
\end{align*}
Taking expectation of the $2p$-th power of the above expression and using Proposition~\ref{multWbdd}, we get
$$\E[\gnorm{\int_s^t S(t-r)V(x)u(r)\,\dd r}^{2p}]\leq C|t-s|^{2p}.$$
Finally, we have
\begin{align*}
\E\left[\gnorm{\int_s^t S(t-r)u(r)\,\dd W(r)}^{2p}\right]&\leq C\E\left[\left(\int_{s}^{t}\norm{u(r)Q^{1/2}}_{\mathcal{L}_2^\sigma}^2\,\dd r\right)^p\right] \\
&\leq C|t-s|^p\E[\sup_{s\leq r\leq t}\gnorm{u(r)}^{2p}]\\
&\leq C|t-s|^p.
\end{align*}
\end{proof}

\subsection{Error analysis for multiplicative noise}
As before, we let $N>0$ be an integer and divide the interval $[0,T]$ 
into subintervals $0=t_0<t_1<\ldots<t_{N-1}<t_N=T$ of equal length $k$ so that $t_n=nk$. 
We discretise the mild equation and our exponential integrator now reads
\begin{align}\label{schMult}
u^n &= S(k)u^{n-1}-\ii kS(k)V(x)u^{n-1}-\ii S(k)u^{n-1}\Delta W^{n-1}\nonumber\\
&=S(t_n)u_0-\ii\sum_{j=0}^{n-1}\int_{t_j}^{t_{j+1}}S(t_n-t_j)V(x)u^j\,\dd r-\ii\sum_{j=0}^{n-1}\int_{t_j}^{t_{j+1}}S(t_n-t_j)u^j\,\dd W(r),
\end{align}
where $\Delta W^{n-1}=W(t_n)-W(t_{n-1})$.

As seen in the theorem below, because of the multiplicative noise, the order of convergence is reduced to one half.
\begin{theorem}\label{LastTh}
Consider the time integration of the stochastic Schr\"{o}dinger equation \eqref{sse-mult} given by the exponential integrator \eqref{schMult}.
Let $\sigma>\frac{d}{2}$ and assume that $\E[\norm{u_0}_{\sigma+1}^{2p}]<\infty$ for some $p\in\mathbb{N}$. 
Assume also that $V(x)\in H^{\sigma+1}$ and $Q^{1/2}\in\mathcal{L}_2^{\sigma+1}$. Then there exists a constant $C$ such that
$$
\E[\max_{n=1,\ldots N}\gnorm{u^n-u(t_n)}^{2p}]\leq Ck^p.
$$
\end{theorem}
\begin{remark}
Following remarks~\ref{remark1} and~\ref{remark2}, the result 
of Theorem~\ref{LastTh} holds with $\sigma=0$ if $d=1,2,3,$ 
$\E[\norm{u_0}_2^{2p}]<\infty$ for some $p\in\mathbb{N}$, $V(x)\in H^2,$ and $Q^{1/2}\in\mathcal{L}_2^2$. 
\end{remark}
\begin{remark}
We could replace the nonlinear term and the $u$ in front of the noise in equation \eqref{sse-mult} with 
more general bounded Lipschitz functions $F\colon H^\sigma\to H^\sigma$ and $\Phi\colon H^\sigma\to \mathcal{L}_2^\sigma$. 
The above result holds with similar assumptions as in Hypothesis~$5.1$ and~$5.2$ in \cite{Weakandstrongorderofconv}, 
that is if we assume that the solution $(u(t))_{t\in[0,T]}$ is in $L^{2p}(\Omega; L^\infty(0,T;H^{\sigma+1})$, $(F(u(t)))_{t\in[0,T]}$ 
is in $L^{2p}(\Omega; L^\infty(0,T;H^{\sigma+1}))$, $(\Phi(u(t)))_{t\in[0,T]}$ is in $L^{2p}(\Omega; L^\infty(0,T;\mathcal{L}_2^{\sigma+1}))$, and $u_0\in L^{2p}(\Omega; H^{\sigma+1})$. 
\end{remark}

\begin{proof}
Using the mild formulation, we have
\begin{align*}
u^n-u(t_n)&=\ii\sum_{j=0}^{n-1}\int_{t_j}^{t_{j+1}}(S(t_n-r)V(x)u(r)-S(t_n-t_j)V(x)u^j)\,\dd r\\
&\quad+\ii\sum_{j=0}^{n-1}\int_{t_j}^{t_{j+1}}(S(t_n-r)u(r)-S(t_n-t_j)u^j)\,\dd W(r)\\
&=: \text{Err}^n_V+\text{Err}^n_W.
\end{align*}
We add and subtract suitable terms to get
\begin{align*}
\text{Err}^n_V&=\ii\sum_{j=0}^{n-1}\int_{t_j}^{t_{j+1}}S(t_n-r)V(x)(u(r)-u(t_j))\,\dd r\\
&\quad+\ii\sum_{j=0}^{n-1}\int_{t_j}^{t_{j+1}}(S(t_n-r)-S(t_n-t_j))V(x)u(t_j)\,\dd r\\
&\quad+\ii\sum_{j=0}^{n-1}\int_{t_j}^{t_{j+1}}S(t_n-t_j)V(x)(u(t_j)-u^j)\,\dd r\\
&=:I^n_1+I^n_2+I^n_3,
\end{align*}
and
\begin{align*}
\text{Err}^n_W&=\ii\sum_{j=0}^{n-1}\int_{t_j}^{t_{j+1}}S(t_n-r)(u(r)-u(t_j))\,\dd W(r)\\
&\quad+\ii\sum_{j=0}^{n-1}\int_{t_j}^{t_{j+1}}(S(t_n-r)-S(t_n-t_j))u(t_j)\,\dd W(r)\\
&\quad+\ii\sum_{j=0}^{n-1}\int_{t_j}^{t_{j+1}}S(t_n-t_j)(u(t_j)-u^j)\,\dd W(r)\\
&=:J^n_1+J^n_2+J^n_3.
\end{align*}

We now proceed to estimate each of these terms. For $I^n_1$, using again that $[r/k]$ denotes 
the integer part of $r/k$ and then H\"{o}lder's inequality, we have
\begin{align*}
\gnorm{I^n_1}&\leq \sum_{j=0}^{n-1}\int_{t_j}^{t_{j+1}}\gnorm{V(x)(u(r)-u(t_j))}\,\dd r\\
&\leq C\gnorm{V(x)}\sum_{j=0}^{n-1}\int_{t_j}^{t_{j+1}}\gnorm{u(r)-u(t_j)}\,\dd r\\
&\leq C\int_0^{t_n}\gnorm{u(r)-u([r/k]k)}\,\dd r\\
&\leq CT^\frac{2p-1}{2p}\left(\int_0^{t_n}\gnorm{u(r)-u([r/k]k)}^{2p}\,\dd r\right)^\frac{1}{2p},
\end{align*}
so that, using Lemma~\ref{reg},
\begin{align*}
\E[\max_{n=1,\ldots,N}\gnorm{I^n_1}^{2p}]&\leq C\E\left[\int_0^T\gnorm{u(r)-u([r/k]k)}^{2p}\,\dd r\right]\\
&\leq C\sum_{j=0}^{N-1}\int_{t_j}^{t_{j+1}}\E[\gnorm{u(r)-u(t_j)}^{2p}]\,\dd r\\
&\leq C\sum_{j=0}^{N-1}\int_{t_j}^{t_{j+1}}|r-t_j|^p\,\dd r\\
&\leq Ck^p.
\end{align*}
By Lemma~\ref{lemma2} and H\"{o}lder's inequality, we have
\begin{align*}
\gnorm{I^n_2}&\leq\sum_{j=0}^{n-1}\int_{t_j}^{t_{j+1}}\gnorm{S(t_n-t_j)(S(t_j-r)-I)V(x)u(t_j)}\,\dd r\\
&\leq C\sum_{j=0}^{n-1}\int_{t_j}^{t_{j+1}}|t_j-r|^{1/2}\norm{V(x)}_{\sigma+1}\norm{u(t_j)}_{\sigma+1}\,\dd r\\
&\leq Ck^{3/2}\left(\sum_{j=0}^{n-1}\norm{u(t_j)}_{\sigma+1}^{2p}\right)^{\frac{1}{2p}}\left(\sum_{j=0}^{n-1} 1^\frac{2p}{2p-1}\right)^{\frac{2p-1}{2p}}\\
&\leq Ck^{\frac{p+1}{2p}}\left(\sum_{j=0}^{n-1}\norm{u(t_j)}_{\sigma+1}^{2p}\right)^{\frac{1}{2p}},
\end{align*}
so that, by Proposition~\ref{multWbdd},
\begin{align*}
\E[\max_{n=1,\ldots,N}\gnorm{I^n_2}^{2p}]&\leq Ck^{p+1}\sum_{j=0}^{N-1}\E[\norm{u(t_j)}_{\sigma+1}^{2p}]\\
&\leq Ck^p.
\end{align*}
Again using H\"{o}lder's inequality, we have for $I^n_3$
\begin{align*}
\gnorm{I^n_3}&\leq\sum_{j=0}^{n-1}\int_{t_j}^{t_{j+1}}\gnorm{V(x)(u(t_j)-u^j)}\,\dd r\\
&\leq Ck\sum_{j=0}^{n-1}\gnorm{V(x)}\gnorm{u(t_j)-u^j}\\
&\leq Ck^\frac{1}{2p}\left(\sum_{j=0}^{n-1}\gnorm{u(t_j)-u^j}^{2p}\right)^\frac{1}{2p},
\end{align*}
and thus
$$\E[\max_{n=1,\ldots,N}\gnorm{I^n_3}^{2p}]\leq Ck\sum_{j=0}^{N-1}\E[\max_{l=0,\ldots,j}\gnorm{u(t_l)-u^l}^{2p}].$$
When estimating the stochastic integrals we will use the Burkholder inequality and that $[r/k]$ is the integer part of $r/k$. Using also H\"{o}lder's inequality, and Lemma \ref{reg}, we have for $J^n_1$
\begin{align*}
\E[\max_{n=1,\ldots,N}\gnorm{J^n_1}^{2p}]&\leq \E\left[\sup_{0\leq t\leq T}\gnorm{\int_0^t S(t-[r/k]k)(u(r)-u([r/k]k))\,\dd W(r)}^{2p}\right]\\
&\leq C\E\left[\left(\int_0^T\norm{(u(t)-u([t/k]k))Q^{1/2}}_{\mathcal{L}_2^\sigma}^2\,\dd t\right)^p\right]\\
&\leq C\norm{Q^{1/2}}_{\mathcal{L}_2^\sigma}^{2p}\E\left[\left(\left(\int_0^T\gnorm{u(t)-u([t/k]k)}^{2p}\,\dd t\right)^\frac{1}{p} T^\frac{p-1}{p}\right)^p\right]\\
&\leq C\sum_{j=0}^{N-1}\int_{t_j}^{t_{j+1}}\E[\gnorm{u(t)-u(t_j)}^{2p}]\,\dd t\\
&\leq C\sum_{j=0}^{N-1}\int_{t_j}^{t_{j+1}}|t-t_j|^p\,\dd t\\
&\leq Ck^p.
\end{align*}
Similarly for $J^n_2$, we have
\begin{align*}
\E[\max_{n=1,\ldots,N}\gnorm{J^n_2}^{2p}]&\leq\E\left[\sup_{0\leq t\leq T}\gnorm{\int_0^t S(t-[r/k]k)(S([r/k]k-r)-I)u([r/k]k)\,\dd W(r)}^{2p}\right]\\
&\leq C\E\left[\left(\int_0^T\norm{(S([t/k]k-t)-I)u([t/k]k)Q^{1/2}}_{\mathcal{L}_2^\sigma}^2\,\dd t\right)^p\right]\\
&\leq CT^{p-1}\int_0^T\E\left[\norm{(S([t/k]k-t)-I)u([t/k]k)Q^{1/2}}_{\mathcal{L}_2^\sigma}^{2p}\right]\,\dd t\\
&\leq C\sum_{j=0}^{N-1}\int_{t_j}^{t_{j+1}}\E[\norm{S(t_j-t)-I)u(t_j)Q^{1/2}}_{\mathcal{L}_2^\sigma}^{2p}\,\dd t\\
&\leq C\norm{Q^{1/2}}_{\mathcal{L}_2^{\sigma+1}}^{2p}\sum_{j=0}^{N-1}\int_{t_j}^{t_{j+1}} |t_j-t|^p\E[\norm{u(t_j)}_{\sigma+1}^{2p}]\,\dd t\\
&\leq Ck^p.
\end{align*}
For $J^n_3$, we have
\begin{align*}
\E[\max_{n=1,\ldots,N}\gnorm{J^n_3}^{2p}]&\leq C\E\left[\sup_{0\leq t\leq T}\gnorm{\int_0^t S(t-[r/k]k)(u([r/k]k)-u^{[r/k]k})\,\dd W(r)}^{2p}\right]\\
&\leq C\E\left[\left(\int_0^T\norm{(u([t/k]k)-u^{[t/k]k})Q^{1/2}}_{\mathcal{L}_2^\sigma}^2\,\dd t\right)^p\right]\\
&\leq C T^{p-1}\norm{Q^{1/2}}_{\mathcal{L}_2^\sigma}^{2p}\sum_{j=0}^{N-1}\int_{t_j}^{t_{j+1}}\E[\gnorm{u(t_j)-u^j}^{2p}]\,\dd t\\
&\leq Ck\sum_{j=0}^{N-1}\E[\max_{l=0,\ldots,j}\gnorm{u(t_l)-u^l}^{2p}].
\end{align*}
Putting all these estimates together yields
$$\E[\max_{n=1,\ldots,N}\gnorm{u^n-u(t_n)}^{2p}]\leq C_1 k^p+C_2k\sum_{j=0}^{N-1}\E[\max_{l=0,\ldots,j}\gnorm{u^l-u(t_l)}^{2p}],$$
and Gronwall's lemma completes the proof.
\end{proof}

\subsection{Numerical experiments for Schr\"odinger equations with a multiplicative noise}
This subsection illustrates the convergence properties of the stochastic 
exponential integrator \eqref{schMult} when applied to the stochastic partial differential equation 
\eqref{sse-mult} on the interval $[0,2\pi]$ with periodic boundary conditions. 
In these numerical experiments, we set $V(x)=0$ and 
$u_0(x)=\e^{-5(x-\pi)^2}$. We compare the stochastic exponential integrator (SEXP) 
with the Crank-Nicolson scheme (CN) and the semi-implicit Euler-Maruyama scheme (SEM). 

Figure~\ref{fig:msMulti} illustrates the convergence errors of the above numerical methods 
for a noise with covariance operator having the eigenvalues $\lambda_n=1/(1+|n|^{5.1})$ for $n\in\Z$. 
The spatial discretisation is done by a pseudospectral method with $M=2^8$ Fourier modes. 
The rates of mean-square convergence (measured in the $L^2$-norm at the end of the interval 
of integration $[0,0.5]$) of these numerical methods are presented in this figure. 
The expected rate of convergence $\bigo{k^1}$ of the stochastic exponential 
integrator, as stated in Theorem~\ref{LastTh}, can be confirmed. 
Here, the exact solution is approximated by the stochastic exponential method with a very small 
time step $k_{\text{exact}}=2^{-10}$ and $M_{\text s}=750000$ samples 
are used for the approximation of the expected values.

\begin{figure}
\centering
\includegraphics*[height=6.5cm,keepaspectratio]{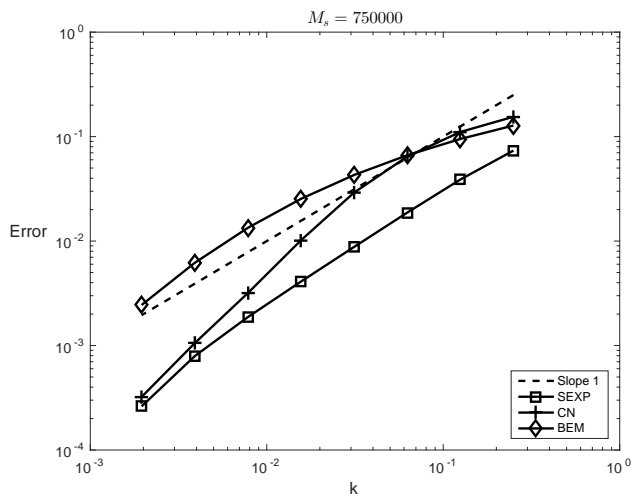}
\caption{Stochastic Schr\"odinger equation with multiplicative noise: 
Mean-square errors for the stochastic exponential integrator (SEXP), 
the Crank-Nicolson scheme (CN), and the semi-implicit Euler-Maruyama (SEM). 
The dotted line has slope $1$. 
\label{fig:msMulti}}
\end{figure}

\bibliographystyle{siam}
\bibliography{biblio}

\end{document}